\PassOptionsToPackage{quiet}{xeCJK}
\documentclass[12pt, twoside]{article}

\usepackage{lmodern}

\usepackage{quiver}

\usepackage{amsmath}
\usepackage{amssymb}
\usepackage{amsthm}
\usepackage{amsopn}
\usepackage{tikz}
\usepackage[indent=20pt]{parskip}
\usepackage{setspace}
\usepackage[margin=1in, a4paper]{geometry}
\usepackage{calc}
\usepackage{hyperref}
\usepackage{authblk}

\hypersetup{
	colorlinks = true,
	linkcolor = black,
	urlcolor = black,
	citecolor = black,
}
\usepackage{fancyhdr}

\newcommand{\mbz}{\mathbb{Z}}

\newcommand{\mbr}{\mathbb{R}}
\newcommand{\mbc}{\mathbb{C}}
\newcommand{\mbs}{\mathbb{S}}

\newcommand{\mbp}{\mathbb{P}}

\newcommand{\mbn}{\mathbb{N}}
\newcommand{\mba}{\mathbb{A}}
\newcommand{\mcf}{\mathcal{F}}
\newcommand{\mcg}{\mathcal{G}}

\newcommand{\mcx}{\mathcal{X}}
\newcommand{\mcy}{\mathcal{Y}}
\newcommand{\mfm}{\mathfrak{m}}

\newcommand{\mco}{\mathcal{O}}
\newcommand{\mcr}{\mathcal{R}}

\newcommand{\llangle}{\langle\!\langle}
\newcommand{\rrangle}{\rangle\!\rangle}
\renewcommand{\>}{\rangle}
\newcommand{\<}{\langle}
\renewcommand{\bar}{\overline}

\DeclareMathOperator{\GL}{GL}
\DeclareMathOperator{\SL}{SL}

\DeclareMathOperator{\Sq}{Sq}

\DeclareMathOperator{\ret}{r\acute{e}t}

\DeclareMathOperator{\supp}{supp}
\DeclareMathOperator{\Spec}{Spec}

\DeclareMathOperator{\Shv}{Shv}
\DeclareMathOperator{\Ab}{Ab}
\DeclareMathOperator{\Id}{Id}
\DeclareMathOperator{\HHom}{Hom}
\DeclareMathOperator{\intHom}{\underline{Hom}}
\DeclareMathOperator{\Gr}{Gr}

\DeclareMathOperator{\Sm}{Sm}
\DeclareMathOperator{\Spc}{Spc}
\DeclareMathOperator{\hSpc}{hSpc}
\DeclareMathOperator{\Gm}{\mathbb{G}_m}

\DeclareMathOperator{\GW}{GW}
\DeclareMathOperator{\W}{W}

\DeclareMathOperator{\I}{I}

\DeclareMathOperator{\cl}{cl}
\DeclareMathOperator{\HH}{H}
\DeclareMathOperator{\K}{K}
\DeclareMathOperator{\tildeK}{\widetilde{K}}
\DeclareMathOperator{\KO}{KO}
\DeclareMathOperator{\BO}{BO}
\DeclareMathOperator{\BSO}{BSO}
\DeclareMathOperator{\M}{M}
\DeclareMathOperator{\MW}{MW}
\DeclareMathOperator{\BM}{BM}
\DeclareMathOperator{\Q}{Q}
\DeclareMathOperator{\RS}{RS}
\DeclareMathOperator{\B}{B\!}

\DeclareMathOperator{\LL}{L}

\DeclareMathOperator{\BGL}{BGL}
\DeclareMathOperator{\BSL}{BSL}

\DeclareMathOperator{\colim}{colim}

\DeclareMathOperator{\Nis}{Nis}
\DeclareMathOperator{\Zar}{Zar}

\DeclareMathOperator{\rr}{r}

\DeclareMathOperator{\ev}{ev}
\DeclareMathOperator{\Vect}{Vect}

\DeclareMathOperator{\CH}{CH}
\DeclareMathOperator{\tildeCH}{\widetilde{CH}}
\DeclareMathOperator{\coker}{coker}

\DeclareMathOperator{\Top}{Top}
\DeclareMathOperator{\alg}{alg}
\DeclareMathOperator{\Sign}{Sign}
\DeclareMathOperator{\Pic}{Pic}
\DeclareMathOperator{\pr}{pr}

\theoremstyle{plain}\newtheorem{thm}{Theorem}[subsection]
\theoremstyle{definition}\newtheorem{defn}[thm]{Definition}
\theoremstyle{plain}
\theoremstyle{plain}\newtheorem{coro}[thm]{Corollary}
\theoremstyle{plain}\newtheorem{lemma}[thm]{Lemma}
\theoremstyle{plain}\newtheorem{prop}[thm]{Proposition}
\theoremstyle{plain}
\theoremstyle{plain}
\theoremstyle{plain}
\theoremstyle{remark}
\theoremstyle{plain}
\theoremstyle{plain}
\theoremstyle{plain}
\theoremstyle{remark}\newtheorem*{remark}{Remark}
\theoremstyle{remark}

\usepackage{amsmath, amssymb}
\usepackage{enumitem}

\usepackage[backend=biber, sorting=nyt]{biblatex}
\title{Algebraizability of Vector Bundles over Real Algebraic
Varieties}
\date{}
\author{Hanqi Wang}
\affil{Qiuzhen College, Tsinghua University}

\begin{document}

\maketitle{}

\pagestyle{fancy}

\fancyhead{}
\fancyhead[CE]{Hanqi Wang}
\fancyhead[CO]{Algebraizability of Vector Bundles over Real Algebraic Varieties}

\begin{abstract}
	Let $X$ be an affine smooth real algebraic variety (in the sense of Bochnak, Coste, and Roy) and let $V$ be a topological vector bundle over $X(\mbr)$. We investigate the problem of deciding whether $V$ is topologically isomorphic to an algebraic vector bundle using motivic homotopy theory. We prove that if $\dim X\leq 3$, then the algebraicity of Stiefel-Whitney classes is a necessary and sufficient condition for $V$ to be algebraizable. Next, we show that when $\dim X=4$ and $X(\mbr)$ is compact, even if the characteristic classes of $V$ are algebraic, there is still an obstruction to algebraizing $V$ related to the Pontryagin class $p_1$ and the Stiefel-Whitney class $w_4$. Then we give some applications of this result. Namely, we give an example where this obstruction is nontrivial, and we investigate the group $\K_0(X)$.
\end{abstract}

\tableofcontents

\section{Introduction}

Let $X$ be an algebraic variety over $\mbr$. Write $X(\mbr)$ for the topological space of $\mbr$-points of $X$ equipped with the Euclidean topology. Let $\Vect_n(X)$ denote the set of isomorphism classes of rank $n$ algebraic vector bundles over $X$, and let $\Vect_n^{\Top}(X(\mbr))$ denote the set of isomorphism classes of topological rank $n$ real vector bundles over $X(\mbr)$. Analytification produces a natural map 
$$\Vect_n(X)\to\Vect_n^{\Top}(X(\mbr)).$$

\begin{defn}
	A vector bundle $V$ over $X(\mbr)$ is said to be \emph{algebraizable} if it lies in the image of this map, i.e. if it is topologically isomorphic to (the $\mbr$-points of) some algebraic vector bundle.
\end{defn}

In this paper, we will study the problem of deciding when a topological vector bundle over $X(\mbr)$ is algebraizable.

\begin{remark}
	Before proceeding, we need to point out a subtlety arising from the fact that $\mbr$ is not algebraically closed. If $X=\Spec A$ is a smooth integral affine variety over $\mbr$, then there are many functions $f\in A$ such that $f$ is not a unit in $A$ but $f$ is nonzero on $X(\mbr)$ (for example, the function $1+x^2$ over $\mba^1_{\mbr}=\Spec\mbr[x]$). We will invert those functions in the structure sheaf of $X$ and we will allow them to be transition functions in vector bundles.
	The schemes obtained by such localizations will be referred to as \emph{real algebraic varieties}.
	In particular, the algebraic vector bundles that we consider in this paper are only defined over a neighbourhood of $X(\mbr)$. This will be made rigorous in \autoref{sec:2.3}.
\end{remark}

\begin{remark}
	For obvious reasons, we will assume that $X$ is irreducible and all vector bundles over $X(\mbr)$ have constant rank.
\end{remark}

This problem, and its analogues concerning complex vector bundles over real or complex varieties, have been studied for a long time. Intuitively, an algebraic vector bundle should have ``algebraic'' characteristic classes in a suitable sense. 
To be more precise, recall that for a variety $Y$ over $\mbc$, we have cycle class maps 
$$\cl:\CH^i(Y)\to\HH^{2i}(Y(\mbc);\mbz)$$
and for a variety $X$ over $\mbr$, we have maps
$$\cl:\CH^i(X)\to\HH^i(X(\mbr);\mbz/2).$$
We will say that a cohomology class is \emph{algebraic} if it lies in the image of the above maps and write $\HH^*_{\alg}(X;\mbz/2)\subseteq\HH^*(X(\mbr);\mbz/2)$ for the subgroup of algebraic classes. One can show that all cohomology classes of Grassmannians are algebraic and the algebraic characteristic classes are mapped to their topological counterparts. As a consequence, algebraic vector bundles have algebraic characteristic classes (Stiefel-Whitney classes $w_*(V)$ in the real case or Chern classes $c_*(V)$ in the complex case).

Now it is very natural to ask if the converse is true, i.e. if a vector bundle has algebraic characteristic classes, is it true that the bundle itself is algebraizable? There are many results on this problem in the case where the dimension is small and in some special cases.

For a smooth projective or smooth affine variety of dimension 3 or less over $\mbc$, the answer is yes \cite{Bnic1987OnCV} \cite{KumarMurthy82}. For an affine smooth real algebraic variety $X$ of dimension 3 or less with $X(\mbr)$ compact and for $X=\mbr\mbp^n$ when $n\leq 13$ and for $X=\mbs^n$, the answer is also yes \cite{Bochnak1989KtheoryOR} \cite{Fossum1969VectorBO} \cite{bd26133d-8ca7-32c5-8280-99dcf151d76b} \cite{algbundonprojspace}. These results are mainly obtained using topological and analytic methods or explicit constructions.
Also, we know that for 4-dimensional varieties, there exist non-algebraizable vector bundles with algebraic characteristic classes \cite{bochnak1989vector} \cite{asok2018obstructionsalgebraizingtopologicalvector}.

In this paper, we will investigate this problem for smooth affine real algebraic varieties using motivic homotopy theory following existing methods developed by Morel, Asok, Fasel, and their collaborators \cite{Mor12A1AT} \cite{Asok_2017affinerep} \cite{asok2018obstructionsalgebraizingtopologicalvector}. For a smooth affine variety $X$, it was shown that $\Vect_n(X)$ is in bijection with the set of $\mba^1$-homotopy classes $[X, \BGL_n]_{\mba^1}$ where $\BGL_n$ is the infinite Grassmannian. Hence, $\Vect_n(X)$, and the obstructions to algebraizing vector bundles, can be described using (relative) Postnikov towers.
In particular, we prove the following theorems.

\begin{thm}[Theorem\autoref{dim3}]
	Let $X$ be an affine smooth real algebraic variety of dimension 3 or less. Then a topological real vector bundle over $X$ is algebraizable if and only if its first and second Stiefel-Whitney classes are algebraic.
\end{thm}

\begin{thm}[Theorem\autoref{dim4}]
	Let $X$ be an affine smooth real algebraic variety of dimension 4 and assume that $X(\mbr)$ is compact. Let $M_1,\cdots,M_p$ be the orientable connected components of $X(\mbr)$, and let $N_1,\cdots,N_q$ be the non-orientable connected components. Then there exist group homomorphisms $k_i:\CH^2(X)\to\HH^4(M_i;\mbz)$ for $1\leq i\leq p$ and maps $k_j':\CH^2(X)\to\HH^4(N_j;\mbz/2)$ for $1\leq j\leq q$ such that the following holds:

	A topological real vector bundle $V$ over $X(\mbr)$ is algebraizable if and only if there exists $c_1\in\CH^1(X),c_2\in\CH^2(X)$ and a class $\alpha\in\HH^3_{\alg}(X;\mbz/2)$ such that
	\begin{enumerate}
		\item $\cl(c_1)=w_1(V),\cl(c_2)=w_2(V)$,
		\item for each $1\leq i\leq p$, $$p_1(V|_{M_i})=k_i(c_2+c_1^2),$$
		\item for each $1\leq j\leq q$, $$(w_4(V)+w_1(V)w_3(V))|_{N_j}=\alpha|_{N_j}\cup w_1(N_j)+k'_j(c_2+c_1^2).$$
	\end{enumerate}
\end{thm}

We will also give some applications.

\begin{thm}[Theorem\autoref{mainexam}]
	Let $Y$ be a smooth affine real algebraic variety of dimension 3 such that $Y(\mbr)$ is compact and orientable. Let $X=Y\times\mbs^1$, where $\mbs^1$ is the standard circle. Then a topological real vector bundle $V$ over $X(\mbr)$ is algebraizable if and only if $w_*(V)\in\HH^*_{\alg}(X;\mbz/2)$ and $p_1(V)=0$.
\end{thm}

\begin{thm}[Proposition\autoref{prop:4.2.3} and Theorem\autoref{thm:4.2.5}]
	Let $X$ be an affine smooth real algebraic variety of dimension 4 and assume that $X(\mbr)$ is compact. Consider the abelian group
	$$G=1\oplus\HH^1_{\alg}(X;\mbz/2)\oplus\CH^2(X)\subseteq\CH^*(X)$$
	with group operation given by multiplication, and let
	$H=w_1(X)\cup\HH^3_{\alg}(X;\mbz/2)$.
	Then there is a short exact sequence
	$$0\to H\to\tildeK_0(X)\to G\to 0.$$
	Moreover, if $X(\mbr)$ is orientable, then we may write
	$\CH^2(X)=\mbz^t\oplus(\mbz/2)^s$
	for some integers $s,t$. Set
	$$b_1=\dim_{\mbz/2}\HH^1_{\alg}(X;\mbz/2), \delta=\dim_{\mbz/2}\{v\in\HH^1_{\alg}(X;\mbz/2)| v^2=0\}.$$
	Then there is an isomorphism
	$$\tildeK_0(X)\cong\mbz^t\oplus (\mbz/2)^{s-b_1+2\delta}\oplus(\mbz/4)^{b_1-\delta}.$$
\end{thm}

This paper is organized as follows.

In \autoref{chap:2}, we recall several basic facts that will be used later. In \autoref{sec:2.1}, we recall some results on the unstable motivic homotopy category. In \autoref{sec:2.2}, we recall the constructions and properties of various real realization maps between cohomology groups. In \autoref{sec:2.3}, we recall basic knowledge on real algebra and analytic methods concerning vector bundles over real algebraic varieties. In particular, we will justify our rather unusual definition of algebraic vector bundles mentioned in the previous remark.

In \autoref{sec:3.1}, we study vector bundles over threefolds, treating the rank 2 case and the rank 3 or more case separately. In \autoref{sec:3.2}, we study vector bundles over compact fourfolds, and we look at some applications.

\section{Preliminaries}\label{chap:2}

\subsection{Motivic Homotopy Theory}\label{sec:2.1}

We first recall motivic homotopy theory, focusing on Morel's work on strongly $\mba^1$-invariant sheaves and relevant computational techniques.

\subsubsection{The Motivic Homotopy Category}

We first recall the construction of the motivic homotopy category from \cite{MV99}.
Fix a Noetherian base scheme $S$ (we are mainly interested in the case where $S=\Spec\mbr$). Let $\Sm_S$ be the category of smooth schemes over $S$.

\begin{defn}
	A map $f:X\to Y$ between schemes is called a \emph{Nisnevich cover} if $f$ is an \'{e}tale cover and for any $y\in Y$, there exists an $x\in X$ such that $f(x)=y$ and $\kappa(x)=\kappa(y)$. Here $\kappa(-)$ denotes the residue field of a point. The Grothendieck topology on $\Sm_S$ generated by Nisnevich covers is called the \emph{Nisnevich topology}.
\end{defn}

Let $\Shv_{\Nis}(\Sm_S,\Spc)$ denote the $\infty$-category of Nisnevich sheaves over $\Sm_S$ with values in $\Spc$, the $\infty$-category of spaces.

\begin{defn}
	A sheaf $\mcf\in\Shv_{\Nis}(\Sm_S,\Spc)$ is \emph{$\mba^1$-invariant} if for all $X\in\Sm_S$ the natural map $\mcf(X)\to\mcf(X\times\mba^1)$ is an equivalence.
\end{defn}

\begin{defn}
	The \emph{$\mba^1$-localization functor} is the localization functor
	$$\LL_{\mba^1}:\Shv_{\Nis}(\Sm_S,\Spc)\to\Shv_{\Nis}(\Sm_S,\Spc)$$
	with respect to the family of morphisms $\{X\times\mba^1\to X\}_{X\in\Sm_S}$.
\end{defn}

\begin{defn}
	The category of \emph{motivic spaces} over $S$, denoted $\Spc(S)$, is the image of the functor $\LL_{\mba^1}$, i.e. the category consisting of $\mba^1$-invariant sheaves.
\end{defn}

For a scheme $X\in\Sm_S$, we use $X$ to denote the motivic space given by the $\mba^1$-localization of its Yoneda sheaf, i.e. $\LL_{\mba^1}\HHom_S(-,X)\in\Spc(S)$.

Observe that $*=S$ is the final object in $\Spc(S)$. As usual, $\Spc_*(S)=\Spc(S)_{*/}$ is the category of \emph{pointed motivic spaces}, and $\hSpc(S),\hSpc_*(S)$ are the corresponding homotopy categories. We will also use the notation $[-,-]_{\mba^1}$ (resp. $[-,-]_{\mba^1,*}$) to denote the set of homotopy classes of maps between unpointed (resp. pointed) motivic spaces.

\subsubsection{$\mba^1$-invariant Sheaves}

From now on, we focus on the special case $S=\Spec k$.
All sheaf cohomology groups are taken under the Nisnevich topology unless otherwise stated. We will recall Morel's work on $\mba^1$-invariant sheaves \cite{Mor12A1AT}.

We say that a Noetherian scheme $X$ over $k$ is \emph{essentially smooth} if $X$ is a cofiltered limit of smooth schemes over $k$ with \'{e}tale affine transition map. For a Nisnevich sheaf $\mcf$ over $\Sm_k$, one can extend it to the category of essentially smooth schemes as follows: for $X=\lim_\alpha X_\alpha$, set $\mcf(X)=\colim_\alpha \mcf(X_\alpha)$. This is indeed well-defined by results from \cite{SGA4Part3}.

For an essentially smooth scheme $X$ over $k$, let $X^{(i)}$ denote the set of points in $X$ of codimension $i$.

\begin{defn}
	Let $\mcf$ be a sheaf of sets over $\Sm_k$. $\mcf$ is said to be \emph{unramified} if
	\begin{enumerate}
		\item for any $X\in\Sm_k$ and any dense open subset $U\subseteq X$, the restriction map $\mcf(X)\to\mcf(U)$ is injective, and
		\item for any irreducible $X\in\Sm_k$ with function field $F=k(X)$, the injective map $$\mcf(X)\to\bigcap_{x\in X^{(1)}}\mcf(\mco_{X,x})$$ is bijective.
	\end{enumerate}
\end{defn}

Morel showed that to describe an unramified sheaf, it suffices to describe its values over finitely generated fields over $k$, its values over discrete valuation rings in such fields, and certain residue (specialization) maps with certain compatibility.

\begin{defn}
	\begin{enumerate}
		\item Let $\mcg$ be a sheaf of groups. It is \emph{strongly $\mba^1$-invariant} if for any $X\in\Sm_k$, the map $$\HH^i(X,\mcg)\to\HH^i(X\times\mba^1,\mcg)$$ induced by the projection is an isomorphism for $i=0,1$.
		\item Let $M$ be a sheaf of abelian groups. It is \emph{strictly $\mba^1$-invariant} if for any $X\in\Sm_k$, the map $$\HH^i(X,M)\to\HH^i(X\times\mba^1,M)$$ induced by the projection is an isomorphism for all $i\in\mbn$.
	\end{enumerate}
\end{defn}

Morel gave a list of axioms that characterizes strong $\mba^1$-invariance of sheaves of groups.

\begin{defn}
	Let $\mcg$ be a sheaf of groups. Its \emph{contraction}, denoted $\mcg_{-1}$, is the sheaf of groups given by
	$$X\mapsto \ker(\mcg(X\times\Gm)\xrightarrow{\ev_1}\mcg(X)).$$
\end{defn}

\begin{remark}
	Intuitively, $\mcg_{-1}$ is the $\Gm$-loop space of $\mcg$.
\end{remark}

Here are some important properties of the above definitions.

\begin{thm}
	If $M$ is a strongly $\mba^1$-invariant sheaf of abelian groups, it is strictly $\mba^1$-invariant.
\end{thm}

\begin{remark}
	The above theorem is proved using Rost-Schmid complexes, which will be introduced in \autoref{subsec:2.1.4}.
\end{remark}

\begin{prop}
	If $M$ is a strictly $\mba^1$-invariant sheaf, then $M$ is unramified.
\end{prop}

\begin{thm}
	Let $\Ab(k)$ denote the category of sheaves of abelian groups over $\Sm_k$. Let $\Ab_{\mba^1}(k)$ denote the full subcategory of $\Ab(k)$ containing strictly $\mba^1$-invariant sheaves. Then $\Ab_{\mba^1}(k)$ is an abelian category and the inclusion $\Ab_{\mba^1}(k)\subseteq \Ab(k)$ is exact.
\end{thm}

\begin{prop}
	If $\mcg$ is strongly $\mba^1$-invariant, then $\mcg_{-1}$ is also strongly $\mba^1$-invariant.
\end{prop}

\begin{prop}
	The functor $\mcg\mapsto\mcg_{-1}$ is exact.
\end{prop}

\subsubsection{Milnor-Witt $K$-theory and Related Sheaves}

Here we define various strongly $\mba^1$-invariant sheaves, mainly following \cite{Mor12A1AT}.

\begin{defn}
    Let $F$ be a field. The \emph{Milnor-Witt K-theory} of $F$, denoted $\K^{\MW}_*(F)$, is the graded associative ring generated by symbols $[u]\in F^{\times}$ in degree 1 and an element $\eta$ in degree $-1$ subject to the following relations.
    \begin{enumerate}
        \item (Steinberg relation) For $u\in F-{0,1}$, we have $[u].[1-u]=0$.
        \item For $a,b\in F^\times$, we have $[ab]=[a]+[b]+\eta.[a].[b]$.
        \item $\eta$ is in the center of $\K^{\MW}_*$.
        \item Set $h:=2+\eta.[-1]$. Then $h.\eta=0$.
    \end{enumerate}
\end{defn}

Recall that for a field $F$ one can define the \emph{Grothendieck-Witt ring}, denoted $\GW(F)$, as the group completion of the monoid of isomorphism classes of nondegenerate quadratic forms over $F$. Its quotient by the hyperbolic form $h:=1+\<-1\>$ is called the \emph{Witt ring}, denoted $\W(F)$. Taking the rank of a form induces a morphism $\W(F)\to\mbz/2$, whose kernel is the \emph{fundamental ideal}, denoted $\I(F)$. Let $\I^n(F)$ denote the $n$-th power of $\I(F)$, and we set $\I^n(F)=\W(F)$ for $n\leq 0$. Let $\I^*(F)=\oplus_{n\in\mbz}\I^n(F)$. This is a $\mbz$-graded ring. 

\begin{prop}
	The map $\GW(F)\to\K^{\MW}_0(F)$ given by $\<a\>\mapsto 1+\eta.[a]$ is an isomorphism.
\end{prop}

From now on, for $a\in F^\times$, we write $\<a\>:=1+\eta.[a]\in\K^{\MW}_0(F)$.

\begin{prop}
	The maps $\eta^n:\K_0^{\MW}(F)\to\K_{-n}^{\MW}(F)$ induce isomorphisms $$\W(F)=\K_0^{\MW}(F)/h\cong \K_{-n}^{\MW}(F).$$
\end{prop}

\begin{thm}
	Consider the map $\K_*^{\MW}(F)\to\I^*(F)$ defined by $[a]\mapsto \<a\>-1$ and $\eta$ acts on $\I^*(F)$ by inclusions $\I^*(F)\subseteq\I^{*-1}(F)$. This map induces an isomorphism $\K^{\MW}_*/h\cong\I^*(F)$.
\end{thm}

\begin{defn}
	The \emph{Milnor $K$-theory} of $F$, denoted $\K^{\M}_*(F)$, is defined as the quotient $\K^{\M}_*(F):=K_*^{\MW}(F)/\eta$.
\end{defn}

\begin{remark}
	From the definitions one easily sees that $\K^{\M}_*(F)$ is the quotient of the tensor algebra of $F^\times$ over $\mbz$ by the Steinberg relation.
\end{remark}

\begin{thm}
	Let $M_*$ denote one of $\K^{\MW}_*$, $\K^{\M}_*$, and $\I^*$. Then $M_*$ admits natural residue maps which define a $\mbz$-graded unramified sheaf. We will use the same notation as $M_*$ for the corresponding unramified sheaf. Furthermore, this sheaf is strongly $\mba^1$-invariant. We also have $(M_*)_{-1}=M_{*-1}$.
\end{thm}

\begin{remark}
	Let $\K_i^{\Q}$ be the Nisnevich sheaf associated to the presheaf $U\mapsto \K_i(U)$ where $K_i(U)$ is the $i$-th $K$-group of $U$. We will call these sheaves \emph{Quillen $K$-theory sheaves}. Using the algebraic Bott periodicity (\cite{voe02nordfjordeidlect}, Part 3, Section 3.2) and applying Theorem\autoref{thm:2.1.5.2} to $\mcx=\BGL_{\infty}$, one can show that the above theorem remains true for $M_*=\K^{\Q}_*$.
\end{remark}

The following theorem follows from Orlov, Vishik, and Voevodsky's proof of Milnor's conjecture \cite{OVV96}.

\begin{thm}
	We have an isomorphism of sheaves $\K^{\M}_n/2\cong\I^n/\I^{n+1}$ given by $$[a_1,\cdots,a_n]\mapsto\llangle a_1,\cdots,a_n\rrangle$$ when $n\geq 1$. Here $\llangle a_1,\cdots,a_n\rrangle=\<1,-a_1\>\otimes\cdots\otimes\<1,-a_n\>$ is the Pfister form.
\end{thm}

We also need to define the twist of a sheaf by a line bundle. Let $X$ be an essentially smooth scheme over $k$ and let $L$ be a line bundle over $X$. Recall that sheaves of the form $A_{-1}$ admits natural $\Gm$-actions. Then we define the sheaf $A_{-1}(L)$ over the small Nisnevich site of $X$ as the sheaf associated to the presheaf
$$(u:U\to X)\mapsto A_{-1}(U)\otimes_{\mbz[\mco(U)^\times]} \mbz[(u^*L)^\times(U)]$$
where $L^\times$ is the $\Gm$-torsor given by the complement of the zero section of $L$.

\begin{prop}
	Let $X$ be an essentially smooth scheme over $k$ and let $L$ be a line bundle over $X$. Then there are exact sequences
	$$0\to\I^{n+1}(L)\to\K_n^{\MW}(L)\to\K^{\M}_n\to0$$
	and
	$$0\to\I^{n+1}(L)\to\I^n(L)\to\K^{\M}_n/2\to0$$
	of sheaves over the small Nisnevich site of $X$.
\end{prop}

\subsubsection{The Rost-Schmid Complex}\label{subsec:2.1.4}

The Rost-Schmid complex provides an explicit chain complex that computes the cohomology of a strongly $\mba^1$-invariant sheaf of abelian groups (\cite{Mor12A1AT}, Chapter 5).

Let $X$ be an essentially smooth scheme over $k$ and let $M$ be a strongly $\mba^1$-invariant sheaf of abelian groups. Recall that $X^{(i)}$ denote the set of points of $X$ with codimension $i$. We first define the abelian groups occurring in the complex.

\begin{defn}
	For any integer $n\geq 1$, set
	$$C^n_{\RS}(X,M):=\bigoplus_{x\in X^{(n)}} M_{-n}(\kappa(x))\otimes_{\mbz[\kappa(x)^\times]}\mbz[\omega_{x/X}^\times]$$
	where $\omega_{x/X}$ denotes the top wedge of the $\kappa(x)$-vector space $\mfm_x/\mfm_x^2$, i.e. the fiber of the canonical bundle of $X$ at $x$, and the action of $\kappa(x)^\times$ on $M_{-n}(\kappa(x))$ is the natural one for $n\geq 1$. When $n=0$, set $C_{\RS}^0(X,M)=M(k(X))$.
\end{defn}

\begin{remark}
	Note that $\omega_{x/X}$ is isomorphic to $\kappa(x)$.
	The twist by $\omega_{x/X}$ is to get rid of choices of uniformizers in some discrete valuation rings when we define the differentials.
\end{remark}

Here we briefly sketch how the differentials are defined. Take $y\in X^{(n-1)},z\in X^{(n)}$ for some $n\geq 1$. Let $Y=\bar{\{y\}}$. Assume $Y$ is normal for the moment. To define the differentials, it suffices to define a map
$$\partial^y_z:M_{-n+1}(\kappa(y))\to M_{-n}(\kappa(z))\otimes_{\mbz[\kappa(z)^\times]}\mbz[\omega^\times_{z/Y}]$$
when $z\in Y$ (otherwise set $\partial_z^y=0$). We set $\partial_z^y$ to be the composition
$$M_{-n+1}(\kappa(y))\xrightarrow{\partial}\HH^1_z(Y,M_{-n+1})\cong M_{-n}(\kappa(z))\otimes_{\mbz[\kappa(z)^\times]}\mbz[\omega^\times_{z/Y}]$$ where $\partial$ is the boundary map in the long exact sequence of cohomology with supports, and the isomorphism essentially follows from homotopy purity (\cite{MV99}, Theorem 2.23).

\begin{remark}
	When $Y$ is not normal, we need to pass to the normalization and use certain transfer maps to deal with the possible field extensions appearing in the normalization of $Y$.
\end{remark}

\begin{defn}
	Let
	$$\partial_{RS}:C^*_{\RS}(X,M)\to C^{*+1}_{\RS}(X,M)$$
	denote the map obtained by taking the direct sum of the differentials of the form $\partial_z^y$. The resulting chain of abelian groups is called the \emph{Rost-Schmid complex}.
\end{defn}

\begin{thm}
	$\partial_{\RS}^2=0$. In other words, the Rost-Schmid complex is indeed a chain complex.
\end{thm}

\begin{thm}
	The Rost-Schmid complex is $\mba^1$-invariant: for any essentially smooth scheme $X$ over $k$ and any strongly $\mba^1$-invariant sheaf of abelian groups $M$, we have a quasi-isomorphism $C^*_{\RS}(X,M)\cong C^*_{\RS}(X\times\mba^1,M)$.
\end{thm}

\begin{thm}
	Let $M$ be a strongly $\mba^1$-invariant sheaf of abelian groups. For an essentially smooth scheme $X$ over $k$, the complex $C_{\RS}^*(-,M)$ viewed as a complex of sheaves over $X$ is an acyclic resolution of $M$ over $X$ in the Zariski and the Nisnevich topology.
	
	In particular, for any essentially smooth scheme $X$ over $k$ and any strongly $\mba^1$-invariant sheaf of abelian groups $M$, there are isomorphisms
	$$\HH^*(C^*_{\RS}(X,M))\cong\HH^*_{\Nis}(X,M)\cong\HH_{\Zar}^*(X,M).$$
\end{thm}

\begin{coro}
	A sheaf of abelian groups is strictly $\mba^1$-invariant if and only if it is strongly $\mba^1$-invariant.
\end{coro}

\begin{remark}
	For a line bundle $L$ over $X$ and a strongly $\mba^1$-invariant sheaf $M$ of abelian groups of the form $M=A_{-1}$, we can define the \emph{twisted Rost-Schmid complex} $C^*_{\RS}(X,M(L))$ by replacing $\omega_X$ with $\omega_X\otimes L$ and twisting the differentials by nonzero local sections accordingly.
\end{remark}

\subsubsection{Homotopy Sheaves and Postnikov Towers}

Now we review the basic properties of homotopy sheaves and Postnikov towers (\cite{Mor12A1AT}, Chapter 6 and Appendix B). Fix a motivic space $\mcx\in\Spc(k)$.

\begin{defn}
	\begin{enumerate}
		\item Define $\pi_0^{\mba^1}(\mcx)$ to be the sheaf associated to the presheaf $U\mapsto[U,\mcx]_{\mba^1}$.
		\item If $\mcx$ is pointed, define $\pi_n^{\mba^1}(\mcx)$ to be the sheaf associated to the presheaf $U\mapsto [\Sigma^n(U_+),\mcx]_{\mba^1,*}$.
	\end{enumerate}
\end{defn}

\begin{thm}\label{thm:2.1.5.2}
	For any pointed motivic space $\mcx\in\Spc_*(k)$, $\pi_1^{\mba^1}(\mcx)$ is strongly $\mba^1$-invariant. As a consequence, for $n\geq 2$, $\pi_n^{\mba^1}(\mcx)$ is strictly $\mba^1$-invariant.
\end{thm}

We need the following computational result.
Recall that if we set 
$$\BGL_n:=\colim(\cdots\to \Gr_{n,m}\to\Gr_{n,m+1}\to\cdots)$$
and
$$\BGL_\infty:=\colim(\cdots\to \BGL_n\to\BGL_{n+1}\to\cdots)$$
then $K=\mbz\times\BGL_{\infty}$ represents algebraic $K$-theory (\cite{voe02nordfjordeidlect}, Part 3, Section 3.2). So $\pi_i^{\mba^1}(\BGL_{\infty})=\K^{\Q}_i$ for $i\geq 1$. By considering the fiber sequences
$$\GL_n\to\GL_{n+1}\to \GL_{n+1}/\GL_n= \mba^{n+1}-\{0\}$$
and arguing as in the classical topological case, we obtain the following proposition.

\begin{prop}
	We have natural maps $\pi_i^{\mba^1}(\GL_n)\to\K^{\Q}_i$ which is an isomorphism for $i\leq n-2$ and an epimorphism for $i=n-1$.
\end{prop}

Next we look at relative Postnikov towers. Fix a morphism $f:\mcx\to\mcy$ between motivic spaces. Assume that $\mcx$ and $\mcy$ is connected. Let $F$ denote the fiber of $f$. Let $P^{(n)}(f)$ denote the $n$-image of $f$, i.e. the space equipped with morphisms $$\mcx\xrightarrow{g_n}P^{(n)}(f)\xrightarrow{h_n}\mcy$$ such that $g_n$ induces an isomorphism on $\pi_i^{\mba^1}$ for $i<n$ and induces an epimorphism on $\pi_n^{\mba^1}$, while $h_n$ induces an isomorphism on $\pi_i^{\mba^1}$ for $i>n$ and induces a monomorphism on $\pi_n^{\mba^1}$. This produces the \emph{relative Postnikov tower} $$\mcx=P^{(\infty)}(f)\to\cdots\to P^{(2)}(f)\to P^{(1)}(f)\to P^{(0)}(f)=\mcy.$$
Furthermore, each $P^{(i+1)}(f)\to P^{(i)}(f)$ fits into a Cartesian square
\[\begin{tikzcd}[ampersand replacement=\&]
	{P^{(i+1)}(f)} \& {\B\pi^{\mba^1}_1(P^{(i)}(f))} \\
	{P^{(i)}(f)} \& {K^{\pi^{\mba^1}_1(P^{(i)}(f))}(\pi_{i+1}^{\mba^1}(F),i+2)}
	\arrow[from=1-1, to=1-2]
	\arrow[from=1-1, to=2-1]
	\arrow[from=1-2, to=2-2]
	\arrow[from=2-1, to=2-2]
\end{tikzcd}\]
where $K^{\pi^{\mba^1}_1(P^{(i)}(f))}(\pi_{i+1}^{\mba^1}(F),i+2)$ is a twisted Eilenberg-Maclane space that fits into a fiber sequence
$$K(\pi_{i+1}^{\mba^1}(F),i+2)\to K^{\pi^{\mba^1}_1(P^{(i)}(f))}(\pi_{i+1}^{\mba^1}(F),i+2)\to \B\pi^{\mba^1}_1(P^{(i)}(f))$$
given by the action of $\pi_1^{\mba^1}$ on $\pi_n^{\mba^1}$.

As in the case of topological spaces, the bottom map in the square above is called the \emph{$k$-invariant}, which is the universal obstruction to lifting a map to $P^{(i)}(f)$ to $P^{(i+1)}(f)$. More precisely, given an essentially smooth scheme $X$ over $k$ and a map $X\to P^{(i)}(f)$, the obstruction to lifting to a map $X\to P^{(i+1)}(f)$ is a cohomology class in $\HH^{i+2}(X,\pi_{i+1}^{\mba^1}(F)_\lambda)$ where $\lambda:X\to\B\pi_1^{\mba^1}(P^{(i)}(f))$ is the map induced from $X\to P^{(i)}(f)$ and $\pi_{i+1}^{\mba^1}(F)_\lambda$ denotes the corresponding twisted sheaf. See \cite{Mor12A1AT} Appendix B for more details.

The following theorem is well-known.

\begin{thm}\label{cohdim}
	Let $X$ be a Noetherian scheme of dimension $d$. Then the small Nisnevich site over $X$ has homotopy dimension $\leq d$.
\end{thm}

As a consequence, the obstruction classes automatically vanish beyond the dimension of $X$.

\subsubsection{Vector Bundles and Characteristic Classes}

We will heavily use the representability of vector bundles over affine schemes, proved by Asok, Hoyois, and Wendt (\cite{Asok_2017affinerep}, Theorem 1).

\begin{thm}\label{affrep}
	Let $X$ be a smooth affine scheme over a field $k$. Then there is a natural bijection of sets
	$$\Vect_n(X)\cong[X,\BGL_n]_{\mba^1}.$$
\end{thm}

\begin{remark}
	The above theorem is based on partial results on the Bass-Quillen conjecture, which states that $\Vect_n$ is $\mba^1$-invariant for affine schemes. In fact, this theorem is true for any smooth affine scheme over a Noetherian ring that is regular over a Dedekind domain with perfect residue fields.
\end{remark}

We also need some characteristic classes. It is well-known that the Chow ring of $\BGL_n$ is the polynomial algebra generated by the Chern classes $c_1,\cdots,c_n$, where $c_i$ lies in degree $i$. Next, we recall that there are algebraic analogues of Pontryagin classes and Euler classes as elements in the Chow-Witt ring of $\BSL_n$. This ring can be described as follows (\cite{Hornbostel_2019}, Theorem 1.3).

\begin{thm}
	We work over a perfect field $k$ of characteristic not 2.
	There is a Cartesian square
	\[\begin{tikzcd}[ampersand replacement=\&]
		{\tildeCH^*(\BSL_n)} \& {\ker\partial} \\
		{\HH^*(\BSL_n,\I^*)} \& {\CH^*(\BSL_n)/2}
		\arrow[from=1-1, to=1-2]
		\arrow[from=1-1, to=2-1]
		\arrow[from=1-2, to=2-2]
		\arrow[from=2-1, to=2-2]
	\end{tikzcd}\]
	where $\partial:\CH^*(\BSL_n)\to\HH^{*+1}(\BSL_n,\I^{*+1})$ is the integral Bockstein map, and
	$$\ker\partial=\mbz[c_{2i+1},2c_{2i_1}\cdot c_{2i_k},c_{2i}^2,c_n]\subseteq \mbz[c_2,\cdots,c_n]=\CH^*(\BSL_n).$$
	The ring $\HH^*(\BSL_n,\I^*)$ is the quotient of the polynomial ring
	$$\W(k)[p_1,p_2,\cdots,p_{\lfloor (n-1)/2\rfloor},e_n,\{\beta_J\}_J]$$
	where $e_n$ is the Euler class, $p_i$ is the $i$-th Pontryagin class, the index set $J$ runs through sets of natural numbers $\{j_1,\cdots,j_l\}$ with $0<j_1<\cdots<j_l\leq\lfloor (n-1)/2\rfloor$, and $\beta_J=\partial(c_{2j_1}\cdots c_{2j_l})$, by the following relations \begin{enumerate}
		\item $\I(k)\beta_J=0$,
		\item if $n=2k+1$ is odd then $e_{n}=\beta_{\{k\}}$, and
		\item for two index sets $J$ and $J'$, $$\beta_J\beta_{J'}=\sum_{k\in J} \beta_{\{k\}}p_{(J\setminus\{k\})\cap J'}\beta_{(J\setminus\{k\})\Delta J'}$$ where $p_A=\prod_{i=1}^l p_{a_i}$ for any index set $A=\{a_1,\cdots,a_l\}$.
	\end{enumerate}
\end{thm}

The Pontryagin classes $p_i\in \tildeCH^{4i}(\BSL_n)$ are given by
$$p_i=(p_i,c_{2i}^2+2\sum_{j=\max\{0,4i-n\}}^{2i-1}(-1)^j c_j c_{4i-j})$$
and the Euler class $e_n\in\tildeCH^n(\BSL_n)$ is given by $e_n=(e_n,c_n)$.

\begin{remark}
	The characteristic classes produce maps
	$$c_i:\BGL_n\to K(\K^{\M}_i,i)$$
	and
	$$p_i:\BSL_n\to K(\K^{\MW}_{4i},4i).$$
	It can be shown that the map $\pi_i(\BGL_{\infty})=\K^{\Q}_i\to\pi_i(K(\K^{\M}_i,i))=\K^{\M}_i$ induced by $c_i$ is the Suslin-Hurewicz map (cf. \cite{asok2019motivicspheresimagesuslinhurewicz}). In particular, its $(i-2)$-th contraction $\K^{\Q}_2\cong\K^{\M}_2\to\K^{\M}_2$ is the multiplication by $(i-1)!$.
\end{remark}

\begin{remark}
	Technically, the Euler class that we need is twisted and lives in the cohomology of $\BGL_n$. This cohomology ring (including the twisted version) is computed by Wendt \cite{Wendt_2024}.
\end{remark}

\subsection{Signature Maps and Cycle Class Maps}\label{sec:2.2}

In this section, we recall the constructions of various real realization maps using signatures of quadratic forms.

We first recall Scheiderer's work on real \'{e}tale cohomology \cite{scheiderer2006ret}.

\begin{defn}
	Let $X$ be a scheme. The \emph{real spectrum} of $X$, denoted $X_{\rr}$, is the set of pairs $(x,\leq_x)$ where $x\in X$ is any point and $\leq_x$ is an ordering on $\kappa(x)$. It admits a topology with open sets of the form $\{f>0\}$ for a function $f$ defined over some affine open subscheme of $X$. The \emph{support map} is the map $\supp:X_{\rr}\to X$ of topological spaces given by $(x,\leq_x)\mapsto x$.
\end{defn}

\begin{defn}
	Let $f:X\to Y$ be a morphism of schemes. $f$ is said to be a \emph{real \'{e}tale cover} if $f$ is \'{e}tale and $f$ induces an epimorphism $X_{\rr}\to Y_{\rr}$. The Grothendieck topology with covers generated by real \'{e}tale covers is called the \emph{real \'{e}tale topology}.
	The \emph{small real \'{e}tale site} of $X$, denoted $X_{\ret}$, is the category of \'{e}tale schemes over $X$ equipped with the real \'{e}tale topology.
\end{defn}

\begin{thm}
	The category of sheaves over $X_{\rr}$ and over $X_{\ret}$ are naturally equivalent.
\end{thm}

\begin{thm}
	The functor $\supp_*$ from the category of abelian sheaves over $X_{\rr}$ to that over $X_{\Zar}$ is faithful and exact.
\end{thm}

Next we recall Jacobson's work on the total signature map \cite{jacobson16}.
Recall that a \emph{nondegenerate quadratic form} over $X$ is a vector bundle $V\to X$ together with a map $\varphi:V\otimes V\to\mco_X$ of vector bundles which induces an isomorphism $V\cong V^\vee$. Let $\W(X)$ denote the Witt ring of $X$ of quadratic forms. Let $\I(X)$ denote the kernel of the rank morphism, i.e. the fundamental ideal.

\begin{defn}
	We define the \emph{total signature map} $\Sign:W(X)\to \HH^0_{\ret}(X;\mbz)$ as follows: given $\varphi\in W(X)$, the function taking $(x,\leq_x)\in X_{\rr}$ to the signature of $\varphi|_x$ under the ordering $\leq_x$ is a continuous function from $X_{\rr}$ to $\mbz$, hence defines an element $\Sign(\varphi)\in\HH^0(X_{\rr};\mbz)=\HH^0_{\ret}(X;\mbz)$.
\end{defn}

Let $X$ be a smooth scheme over $\mbr$. Assume that $L$ is a line bundle over $X$. Then one can twist the constant sheaf $\mbz$ over $X(\mbr)$ by the line bundle $L$ to obtain a local system $\mbz(L)$ over $X(\mbr)$. It can be shown that there are isomorphisms $\HH^n(X_{\rr};\mbz(L))=\HH^n(X(\mbr);\mbz(L))$ (\cite{asok2025splittingvectorbundlesreal}, Proposition 1.4.2).

\begin{defn}
	The \emph{real cycle class maps} are the maps $\gamma^n_m(L):\HH^n(X,\I^m(L))\to \HH^n(X(\mbr);\mbz(L))$ induced from the map $\Sign/2^m$ twisted by $L$.
\end{defn}

\begin{defn}
	The commutative diagrams
	\[\begin{tikzcd}[ampersand replacement=\&]
		{\I^{n+1}(L)} \& {\I^n(L)} \& {\K^{\M}_n/2} \\
		{a_{\ret}\mbz(L)} \& {a_{\ret}\mbz(L)} \& {a_{\ret}\mbz/2}
		\arrow[from=1-1, to=1-2]
		\arrow["{\Sign/2^{n+1}}", from=1-1, to=2-1]
		\arrow[from=1-2, to=1-3]
		\arrow["{\Sign/2^n}", from=1-2, to=2-2]
		\arrow[from=1-3, to=2-3]
		\arrow["2", from=2-1, to=2-2]
		\arrow[from=2-2, to=2-3]
	\end{tikzcd}\]
	induces maps $\bar{\gamma}^n_m:\HH^n(X,\K^{\M}_m)\to\HH^{n}(X(\mbr);\mbz/2)$, which we call the \emph{mod 2 real cycle class maps}.
\end{defn}

For convenience, we write $\gamma_n(L)=\gamma^n_n(L)$ and $\bar{\gamma}_n=\bar{\gamma}_n^n$.

\begin{defn}
	The \emph{algebraic cohomology groups} of $X$, denoted $\HH^*_{\alg}(X;\mbz/2)$ (resp. $\HH^*_{\alg}(X;\mbz(L)))$, are the image of the cycle class maps $\bar{\gamma}$ (resp. $\gamma(L)$).
\end{defn}

Now the commutative diagram
\[\begin{tikzcd}[ampersand replacement=\&]
	\cdots \& {\I^{n-1}(X)} \& {\I^n(X)} \& {\I^{n+1}(X)} \& \cdots \\
	\cdots \& {\HH_{\ret}^0(X;\mbz)} \& {\HH_{\ret}^0(X;\mbz)} \& {\HH_{\ret}^0(X;\mbz)} \& \cdots
	\arrow[from=1-1, to=1-2]
	\arrow["2", from=1-2, to=1-3]
	\arrow["{\Sign/2^{n-1}}", from=1-2, to=2-2]
	\arrow["2", from=1-3, to=1-4]
	\arrow["{\Sign/2^{n}}", from=1-3, to=2-3]
	\arrow[from=1-4, to=1-5]
	\arrow["{\Sign/2^{n+1}}", from=1-4, to=2-4]
	\arrow[from=2-1, to=2-2]
	\arrow["\Id", from=2-2, to=2-3]
	\arrow["\Id", from=2-3, to=2-4]
	\arrow[from=2-4, to=2-5]
\end{tikzcd}\]
induces a map $\Sign:\colim_{\times 2}\I^n(X)\to\HH_{\ret}^0(X;\mbz)$. 
The following theorems are Jacobson's main theorems \cite{jacobson16}.

\begin{thm}
	The map $\Sign:\colim_{\times2} \I^n\to a_{\ret}\mbz$ is an isomorphism of Nisnevich sheaves. Here $a_{\ret}$ denotes the sheafification functor under the real \'{e}tale topology.
\end{thm}

\begin{thm}\label{thm:2.2.0.10}
	If $X$ is a smooth scheme over $\mbr$ of dimension $d$, then for any $m\geq d+1$ and any integer $n$, the map $\gamma^n_m(L)$ is an isomorphism.
\end{thm}

We also need the following result by Colliot-Th\'{e}l\`{e}ne and Scheiderer \cite{CTS96} and Lerbet \cite{Lerbet_2026}.

\begin{thm}[\cite{Lerbet_2026}, Proposition 3.3]\label{thm:2.2.0.11}
	Let $X$ be a smooth scheme over $\mbr$ of dimension $d$ and let $L$ be a line bundle over $X$. Then the map $\I^d(L)\xrightarrow{2}\I^{d+1}(L)$ is an epimorphism. Moreover, 
	let $\bar{\K}$ denote the kernel of the map $\I^d(L)\xrightarrow{2}\I^{d+1}(L)$. Assume that $X$ is not proper or $X(\mbr)\neq\varnothing$. Then $\HH^d(X,\bar{\K})=0$.
\end{thm}

\begin{remark}
	Consider the map $\CH^n(X)\to\HH^n(X(\mbr);\mbz/2)$ given by $$\CH^n(X)\to\HH_{d-n}^{\BM}(X(\mbr);\mbz/2)\cong\HH^n(X(\mbr);\mbz/2)$$
	where the first map takes a cycle to its fundamental class in Borel-Moore homology and the isomorphism is Poincar\'{e} duality. It turns out that this map coincides with $\bar{\gamma}_n$ (\cite{Lerbet_2026}, Section 2.5).
\end{remark}

Finally, we recall the following theorem, which states that algebraic characteristic classes are mapped to their topological counterparts (\cite{asok2025splittingvectorbundlesreal}, Proposition 1.4.4, and \cite{Hornbostel_2021}, Theorem 6.3).

\begin{thm}\label{thm:2.2.0.12}
	\begin{enumerate}
		\item Under the mod 2 cycle class map $$\bar{\gamma}:\CH^*(\BGL_n)\to\HH^*(\BO_n;\mbz/2)$$ the $i$-th Chern class $c_i$ is mapped to the $i$-th Stiefel-Whitney class $w_i$.
		\item Under the composition $$\tildeCH^*(\BSL_n)\to\HH^*(\BSL_n,\I^*)\xrightarrow{\gamma}\HH^*(\BSO_n;\mbz)$$ the (algebraic) $i$-th Pontryagin class $p_i$ is mapped to the (topological) $i$-th Pontryagin class $p_i$, and the (algebraic) Euler class $e_n$ is mapped to the topological Euler class $e_n$.
	\end{enumerate}
\end{thm}

\begin{remark}
	For any smooth scheme $X$ over $\mbr$ and any line bundle $L$ over $X$, the image of the above composition
	$$\tildeCH^*(X,L)\to\HH^*(X,\I^*(L))\xrightarrow{\gamma(L)}\HH^*(X(\mbr);\mbz(L))$$
	coincides with the image of $\gamma(L)$. Indeed, we have short exact sequences of sheaves
	$$0\to 2\K^{\M}_n\xrightarrow{h}\K^{\MW}_n(L)\to\I^n(L)\to 0$$
	and by the Rost-Schmid complex, $\HH^{n+1}(X,2\K^{\M}_n)=0$ since $(\K^{\M}_n)_{-n-1}=0$. So the map $\tildeCH^*(X,L)\to\HH^*(X,\I^*(L))$ is surjective. We will use this fact implicitly in the sequel.
\end{remark}

\subsection{Real Algebraic Geometry}\label{sec:2.3}

In this section, we explain our unusual conventions about real algebraic geometry and list some existing results about algebraic vector bundles following the textbook by Bochnak, Coste, and Roy \cite{BCR}.

\begin{defn}\label{defn:2.3.0.1}
	Let $Y=\Spec A$ be a scheme of finite type over $\mbr$. Let $S\subseteq A$ be the multiplicative subset consisting of elements which are nonzero over $Y(\mbr)$.
	The \emph{affine real algebraic variety} defined by $Y$ is the scheme $X=\Spec S^{-1}A$. We will write $\mcr(X)=\Gamma(X,\mco_X)$ for the ring of \emph{regular functions} over $X$.
\end{defn}

\begin{remark}
	This definition is slightly different from \cite{BCR} Definition 3.2.9, where they only considered closed points.
\end{remark}

In the setting of Definition\autoref{defn:2.3.0.1}, if $Y$ is smooth, then $X$ is essentially smooth. We will say that an affine real algebraic variety is \emph{smooth} if it comes from a smooth affine scheme of finite type over $\mbr$. Also note that $X_{\rr}=Y_{\rr}$, so the results in the previous sections is valid for $X$.

\begin{remark}
	Write $X=\lim_\alpha U_\alpha$ where $U_\alpha$ are the open subsets of $Y$ containing $X$. From now on, for a motivic space $\mcy$, a map $X\to \mcy$ will refer to an element in the colimit $\colim_\alpha\HHom_{\Spc(k)}(U_\alpha,\mcy)$. Similarly, for a strictly $\mba^1$-invariant sheaf $M$, we set
	$$\HH^i(X,M)=\colim_\alpha\HH^i(U_\alpha,M).$$
	One can check that this coincides with the cohomology of the Rost-Schmid complex $C^*_{\RS}(X,M)$.
\end{remark}

Now we list some results on vector bundles over a smooth affine real algebraic variety $X$ (cf. \cite{BCR}, Chapter 12).
As one would expect, given a projective module of finite rank $M$ over $\mcr(X)$, there exists uniquely (up to isomorphism) a vector bundle $V_{\alg}\to X$ such that $\Gamma(V_{\alg})=M$. Taking $\mbr$-points, we get a topological real vector bundle $V\to X(\mbr)$. Recall that a real vector bundle over $X(\mbr)$ is \emph{algebraizable} if it is topologically isomorphic to the $\mbr$-points of some algebraic vector bundle.

For a topological space $S$, let $\KO^0(S)$ denote the group completion of the monoid of isomorphism classes of real topological vector bundles over $S$. Then we have a natural group homomorphism $\K_0(\mcr(X))\to\KO^0(X(\mbr))$ induced by taking $\mbr$-points.

\begin{thm}[\cite{BCR}, Theorem 12.3.3 and Proposition 12.3.5]\label{thm:2.3.0.2}
	Assume that $X(\mbr)$ is compact.
	\begin{enumerate}
		\item Two algebraic vector bundles over $X$ are algebraically isomorphic if and only if they are topologically isomorphic. In particular, the natural maps $\Pic(\mcr(X))=\CH^1(X)\to\Pic(X(\mbr))=\HH^1(X(\mbr);\mbz/2)$ and $\K_0(\mcr(X))\to\KO^0(X(\mbr))$ are injective.
		\item A topological vector bundle $V\to X(\mbr)$ is algebraizable if and only if its class in $\KO^0(X(\mbr))$ lies in the image of $\K_0(\mcr(X))\to\KO^0(X(\mbr))$.
	\end{enumerate}
\end{thm}

This shows that if $X(\mbr)$ is compact, then it suffices to study algebraizability of stable vector bundles.

\begin{remark}
	Compactness of $X(\mbr)$ is essential in this theorem. Indeed, the proof uses partition of unity and Stone-Weierstrass approximation.
\end{remark}

Finally, we need to show that under our convention on maps out of $X$, the space $\BGL_n$ still represents the functor $\Vect_n$.

\begin{lemma}
	Write $X=\lim_\alpha U_\alpha$ where $U_\alpha$ are the open subsets of $Y$ containing $X$. Then the natural map
	$$\colim_\alpha[U_\alpha,\BGL_n]_{\mba^1}\to\Vect_n(X)$$
	is bijective.
\end{lemma}

\begin{proof}
	Let $Y=\Spec A$ be a smooth affine scheme over $\mbr$ such that $X=\Spec S^{-1} A$ as in Definition\autoref{defn:2.3.0.1}.

	We first show surjectivity. It suffices to show that any vector bundle $V$ of rank $n$ over $X$ can be extended to a vector bundle over some $U_\alpha$. Let $M=\Gamma(V)$ be the associated projective module over $\mcr(X)=S^{-1}A$. Let $N$ be an integer such that $M$ is a direct summand of $\mco_X^N$. Let $P\in \M_{N}(S^{-1}A)$ be the matrix representing the projection map $\mco_X^N\to M$. Then by taking common denominators, we may find an element $f\in A$ such that $P$ is defined in $A_f$. Let $N$ be the $A_f$-module given by the image of $P$. Then the vector bundle associated to $N$ over $\Spec A_f$ is an extension of $V$.

	Next we show injectivity. Let $U\subseteq Y$ be an open subset containing $X$ and let $V,V'$ be two vector bundles over $U$ which are isomorphic when restricted to $X$. Consider the vector bundle $W=\HHom(V,V')$. Let $\GL(V,V')$ denote the open subset of the total space of $W$ consisting of isomorphisms. Then the section $\sigma$ of $W|_X$ given by the isomorphism $V|_X\cong V'|_X$ lies in $\GL(V,V')$. By shrinking $U$ if necessary, we may extend $\sigma$ to a section of $W$ lying in $\GL(V,V')$. This section witnesses the isomorphism between $V$ and $V'$.
\end{proof}

\section{Vector Bundles over Threefolds}\label{sec:3.1}

In this section we assume that $X$ is a smooth affine real algebraic variety of dimension 3 or less.

\subsection{Vector Bundles of Rank 2}

We start with the topological classification of rank 2 real vector bundles.

\begin{prop}
	Let $S$ be a CW-complex. Then rank 2 real vector bundles over $S$ are classified by the first Stiefel-Whitney class $w_1\in\HH^1(S;\mbz/2)$ and the twisted Euler class $e_{w_1}\in\HH^2(S;\mbz_{w_1})$. Here $\mbz_{w_1}$ is the local system obtained by twisting the constant sheaf $\mbz$ by the line bundle given by $w_1$.
\end{prop}

\begin{proof}
	The homotopy groups of $\BO_2$ are the following: $\pi_1(\BO_2)=\mbz/2$, $\pi_2(\BO_2)=\mbz$, and $\pi_i(\BO_2)=0$ for $i\geq 3$. Moreover, the action of $\pi_1(\BO_2)$ on $\pi_2(\BO_2)$ is the only nontrivial action. The conclusion then follows from standard obstruction theory.
\end{proof}

Let $V$ be a rank 2 real vector bundle over $X(\mbr)$, and let $w_1(V),e_{w_1}(V)$ denote the characteristic classes of $V$. Then $V$ is algebraic if and only if we can find an algebraic vector bundle $V_{\alg}\to X$ such that $w_1(V_{\alg})=w_1(V)$ and $e_{w_1}(V_{\alg})=e_{w_1}(V)$.

\begin{prop}\label{prop:3.1.1.2}
	$V$ is algebraizable if and only if 
	\begin{enumerate}
		\item $w_1(V)\in\HH^1_{\alg}(X;\mbz/2)$, and
		\item if 1. is satisfied, there exists an algebraic line bundle $L$ over $X$ such that $w_1(L)=w_1(V)$ and $e_{w_1}(V)\in\HH^2_{\alg}(X;\mbz(L))$.
	\end{enumerate}
\end{prop}

\begin{proof}
	The ``only if'' direction is obvious: take $L=\det V$. We focus on the ``if'' direction. We will construct a vector bundle by constructing a map $X\to\BGL_2$ and applying Theorem\autoref{affrep}.

	The Postnikov tower of $\BGL_2$ looks like
	$$\BGL_2\to\cdots\to\BGL_2^{(2)}\to\BGL_2^{(1)}\to *$$
	with $\BGL_2^{(1)}=\B\Gm$. In view of Theorem\autoref{cohdim}, the obstructions to lifting a map $X\to\BGL_2^{(2)}$ to a map $X\to\BGL_2$ vanish since they lie in $\HH^{i+1}(X,\pi_i^{\mba^1}(\BGL_2)),i\geq 3$ which are zero. So there is a surjection
	$$[X,\BGL_2]_{\mba^1}\twoheadrightarrow [X,\BGL_2^{(2)}]_{\mba^1}.$$
	We know how to describe the space $\BGL_2^{(2)}$. $\pi_2^{\mba^1}(\BGL_2)=\K^{\MW}_2$ and the action of $\pi_1^{\mba^1}(\BGL_2)=\Gm$ on $\K_2^{\MW}$ is the restriction of the multiplication action of $\K^{\MW}_0$ on $\K^{\MW}_2$ (\cite{AsokFasel2014threefold}, Proposition 6.3). Therefore $\BGL_2^{(2)}$ is the twisted Eilenberg-Maclane object $K^{\Gm}(\K^{\MW}_2,2)$. So the maps $X\to K^{\Gm}(\K^{\MW}_2,2)$ are classified by $c_1:X\to\B\Gm$ and $e_{c_1}\in\HH^2(X,\K^{\MW}_2(L))$ where $L$ is the line bundle over $X$ corresponding to $c_1$. By Theorem\autoref{thm:2.2.0.12}, we conclude that if the conditions in the above statement are satisfied, the algebraic vector bundle $V_{\alg}$ can be constructed by lifting a map $X\to\BGL_2^{(2)}$ given by (a choice of) preimages of $w_1(V)$ and $e_{w_1}(V)$ under the cycle class maps.
\end{proof}

Next we discuss how to eliminate the twist by $L$. The following theorem is due to Lerbet (\cite{Lerbet_2026}, Proposition 4.4).

\begin{thm}\label{thm:3.1.1.3}
	Let $L$ be a line bundle over $X$. The image of the cycle class map $\gamma_2(L)$ is the inverse image of the image of $\bar{\gamma}_2$ under the mod 2 reduction map.
\end{thm}

\begin{proof}
	Consider the commutative diagram
	\[\begin{tikzcd}[ampersand replacement=\&, column sep=20pt]
		{\HH^2(X,\I^3(L))} \& {\HH^2(X,\I^2(L))} \& {\HH^2(X,\K^{\M}_2/2)} \& {\HH^3(X,\I^3(L))} \\
		{\HH^2(X(\mbr);\mbz(L))} \& {\HH^2(X(\mbr);\mbz(L))} \& {\HH^2(X(\mbr);\mbz/2)} \& {\HH^3(X(\mbr);\mbz(L))}
		\arrow[from=1-1, to=1-2]
		\arrow["{\gamma^2_3}", from=1-1, to=2-1]
		\arrow[from=1-2, to=1-3]
		\arrow["{\gamma_2}", from=1-2, to=2-2]
		\arrow[from=1-3, to=1-4]
		\arrow["{\bar{\gamma}_2}", from=1-3, to=2-3]
		\arrow["{\gamma_3}", from=1-4, to=2-4]
		\arrow["2", from=2-1, to=2-2]
		\arrow[from=2-2, to=2-3]
		\arrow[from=2-3, to=2-4]
	\end{tikzcd}\]
	where the rows are long exact sequences in cohomology associated to short exact sequences of sheaves. By a diagram chase, it suffices to show that $\gamma_3$ is injective and $\gamma_3^2$ is surjective.

	From Theorem\autoref{thm:2.2.0.10} we know that $\HH^i(X(\mbr);\mbz(L))\cong\HH^i(X,\I^4(L))$ for all $i$. The short exact sequence of sheaves
	$$0\to\bar{\K}\to\I^3(L)\xrightarrow{2}\I^4(L)\to0$$
	yields a commutative diagram
	\[\begin{tikzcd}[ampersand replacement=\&,column sep=8pt]
		{\HH^2(X,\I^3(L))} \& {\HH^2(X,\I^4(L))} \& {\HH^3(X,\bar{\K})} \& {\HH^3(X,\I^3(L))} \& {\HH^3(X,\I^4(L))} \\
		\& {\HH^2(X(\mbr);\mbz(L))} \&\&\& {\HH^3(X(\mbr);\mbz(L))}
		\arrow[from=1-1, to=1-2]
		\arrow["{\gamma^2_3}"', from=1-1, to=2-2]
		\arrow[from=1-2, to=1-3]
		\arrow["\cong", from=1-2, to=2-2]
		\arrow[from=1-3, to=1-4]
		\arrow[from=1-4, to=1-5]
		\arrow["{\gamma_3}"', from=1-4, to=2-5]
		\arrow["\cong", from=1-5, to=2-5]
	\end{tikzcd}\]
	with an exact row. By Theorem\autoref{thm:2.2.0.11}, $\HH^3(X,\bar{\K})=0$, so $\gamma_3^2$ is surjective and $\gamma_3$ is injective as required.
\end{proof}

\begin{remark}
	In other words, for any line bundle $L$ over $X$, the preimage of $\HH^2_{\alg}(X;\mbz/2)$ in $\HH^2(X(\mbr);\mbz(L))$ under the mod 2 reduction map is precisely $\HH^2_{\alg}(X;\mbz(L))$.
\end{remark}

Combining these results, we obtain the following theorem.

\begin{thm}\label{dim3rk2}
	Let $V$ be a rank 2 real vector bundle over $X(\mbr)$. Then $V$ is algebraizable if and only if the Stiefel-Whitney classes $w_1(V)$ and $w_2(V)$ lie in $\HH^*_{\alg}(X;\mbz/2)$.
\end{thm}

\begin{proof}
	It suffices to prove the ``if'' direction. Since $w_1(V)$ is algebraic, we may find an algebraic line bundle $L$ over $X$ such that $w_1(L)=w_1(V)$. By Theorem\autoref{thm:3.1.1.3}, since $w_2(V)$ is algebraic and $w_2(V)$ is the mod 2 reduction of $e_{w_1}(V)$, $e_{w_1}(V)$ is also algebraic, i.e. lies in $\HH_{\alg}^2(X;\mbz(L))$. We conclude the proof by applying Proposition\autoref{prop:3.1.1.2}.
\end{proof}

\subsection{Vector Bundles of Rank 3 or More}

We first show that we can reduce the rank $\geq 3$ case to the rank 3 case.

\begin{prop}
	Let $V$ be a rank $r$ topological real vector bundle over $X(\mbr)$ with $r\geq 3$. Then $V$ splits off a trivial bundle of rank $r-3$.
\end{prop}

\begin{proof}
	When $r>3$, a generic section of $V$ is nonzero and generates a trivial line bundle as a direct summand.
\end{proof}

From now on, assume that $V\to X(\mbr)$ is a rank 3 topological real vector bundle.
The topological classification of such bundles is given by the following.

\begin{prop}
	Let $S$ be a CW-complex of dimension 3 or less. Then rank 3 real vector bundles over $S$ are classified by the first and second Stiefel-Whitney classes $w_1$ and $w_2$.
\end{prop}

\begin{proof}
	Consider the map $(w_1,w_2):\BO_3\to K(\mbz/2,1)\times K(\mbz/2,2)$ corresponding to the first and second Stiefel-Whitney classes. Then this map induces isomorphisms on $\pi_i$ for $i\leq 3$ and induces a surjection on $\pi_4$. Since $S$ has dimension $\leq 3$, using a standard obstruction theory argument, we conclude that there is a bijection $[S,\BO_3]\xrightarrow{(w_1,w_2)}[S,K(\mbz/2,1)\times K(\mbz/2,2)]$.
\end{proof}

\begin{remark}
	By Wu's formula, the third Stiefel-Whitney class is given by $w_3=\Sq^1(w_2)+w_1w_2$, so $w_3$ does not appear in the above classification.
\end{remark}

\begin{thm}\label{dim3rk3}
	Let $V$ be a topological rank 3 real vector bundle over $X(\mbr)$. Then $V$ is algebraizable if and only if $w_1(V)$ and $w_2(V)$ are algebraic.
\end{thm}

\begin{proof}
	Again, it suffices to prove the ``if'' direction. We fix a map $$(c_1(V),c_2(V)):X\to K(\K^{\M}_1,1)\times K(\K^{\M}_2,2)$$ such that its image under the real cycle class map is $(w_1(V),w_2(V))$.

	Consider the map $(c_1,c_2):\BGL_3\to K(\K^{\M}_1,1)\times K(\K^{\M}_2,2)$. This map induces isomorphisms on $\pi_1^{\mba^1}$ and $\pi_2^{\mba^1}$, and is a surjection on $\pi_3^{\mba^1}$. Since the cohomological dimension of $X$ is 3 (Theorem\autoref{cohdim}), by a standard obstruction theory argument similar to the proof of Proposition\autoref{prop:3.1.1.2}, we have a surjective map
	$$[X,\BGL_3]_{\mba^1}\to[X,K(\K^{\M}_1,1)\times K(\K^{\M}_2,2)]_{\mba^1}$$
	which produces the desired vector bundle $V_{\alg}$ as a map $V_{\alg}:X\to\BGL_3$ lifting the characteristic classes $(c_1(V),c_2(V))$.
\end{proof}

Putting things together, we obtain one of our main theorems.

\begin{thm}\label{dim3}
	Let $X$ be a smooth affine real algebraic variety of dimension 3. Then a topological real vector bundle $V$ over $X(\mbr)$ is algebraizable if and only if its first and second Stiefel-Whitney classes $w_1(V),w_2(V)$ are in $\HH^*_{\alg}(X;\mbz/2)$.
\end{thm}

\begin{proof}
	The statement is trivial if $V$ is a line bundle. If $V$ has rank 2, apply Theorem\autoref{dim3rk2}. If $V$ has rank 3 or more, assume $V$ has rank 3 by splitting off a trivial bundle and apply Theorem\autoref{dim3rk3}.
\end{proof}

\section{Vector Bundles over Compact Fourfolds}\label{sec:3.2}

In this section we assume that $X$ is a smooth affine real algebraic variety of dimension 4 such that the smooth manifold $X(\mbr)$ is compact. Let $M_1,\cdots,M_p$ be the orientable connected components of $X(\mbr)$, and let $N_1,\cdots,N_q$ be the non-orientable connected components.

\subsection{General Theory}

Let $V$ be a topological real vector bundle over $X(\mbr)$. Since $X(\mbr)$ is compact, $V$ is algebraizable if and only if it is stably algebraizable by Theorem\autoref{thm:2.3.0.2}. Therefore, if we assume that $w_1(V)$ is algebraic, $V$ is algebraizable if and only if $V\oplus\det V$ is algebraizable. So we may replace $V$ by $V\oplus\det V$ and assume that $V$ is orientable for the moment.

The topological classification of stable vector bundles over a 4-dimensional CW-complex was obtained by Dold and Whitney \cite{cefe0e4d-d22a-33d8-8a9b-5661eb437142}.

\begin{thm}\label{doldwhitney}
	Let $S$ be a CW-complex of dimension 4. Then stable orientable real vector bundles over $S$ are classified by the second and fourth Stiefel-Whitney classes $w_2$ and $w_4$ and the first Pontryagin class $p_1$.
\end{thm}

Before our main computations, we need the following lemma.

\begin{lemma}\label{lem:3.2.1.2}
	Let $Y$ be a smooth affine real algebraic variety with $\dim Y\geq 2$ and $Y(\mbr)$ compact. Then $\CH_1(Y)$ is 2-torsion.
\end{lemma}

\begin{proof}
	Let $Z\subseteq Y$ be a 1-dimensional subscheme. We wish to prove that the class $[Z]\in\CH_1(Y)$ is 2-torsion. By \cite{kollár2024flatpushforwardschernclasses} Theorem 1.2, we may assume that $Z$ is smooth. By \cite{jelonek14} Theorem 1.1, we may find a smooth 2-dimensional subvariety $S\subseteq Y$ containing $Z$. By Theorem\autoref{thm:2.3.0.2}, the group $\CH^1(S)$ is 2-torsion, so the class $[Z]\in\CH_1(Y)$ is also 2-torsion.
\end{proof}

Let $\BSL_n$ denote the classifying space of $\SL_n$. It is also the homotopy fiber of $\det:\BGL_n\to\B\Gm$. Let $\BSL_{\infty}=\colim_n\BSL_n$. We claim that $V$ is algebraizable if and only if there is a map $X\to\BSL_{\infty}$ such that the class in $\K_0(\mcr(X))$ given by the composition $X\to\BSL_{\infty}\to\BGL_{\infty}$ maps to the class of $V$ in $\KO^0(X(\mbr))$. Indeed, algebraizability is equivalent to stable algebraizability, and if $V$ is isomorphic to an algebraic vector bundle $V_{\alg}$ then $\det V_{\alg}=\det V$ is trivial and the map $X\to\BGL_\infty$ given by $V_{\alg}$ factors through $\BSL_{\infty}$.

Recall that we have maps $c_i:\BSL_\infty\to K(\K^{\M}_i,i)$ and $p_i:\BSL_{\infty}\to K(\K^{\MW}_{4i},4i)$ induced by the characteristic classes. We further take the quotient of $p_i$ by $h$ to obtain maps $p_i:\BSL_\infty\to K(\I^{4i},4i)$.

Now consider the map
$$(c_2,c_4,p_1):\BSL_\infty\to K(\K^{\M}_2,2)\times K(\K^{\M}_4,4)\times K(\I^{4},4)$$
and let $F$ denote its fiber. Then we have an exact sequence
$$\K^{\Q}_4\xrightarrow{f}\K^{\M}_4\oplus\I^4\to\pi_3^{\mba^1}(F)\to\K^{\Q}_3\to 0.$$
To apply obstruction theory, we need to understand the group $\HH^4(X,\pi_3^{\mba^1}(F))$.
Let $A$ be the cokernel of $f$.

\begin{prop}
	We have $\HH^4(X,A)=\HH^4(X(\mbr);\mbz/2)\oplus\HH^4(X(\mbr);\mbz)$.
\end{prop}

\begin{proof}
	Using Theorem\autoref{thm:2.2.0.10} and Theorem\autoref{thm:2.2.0.11}, one can prove that $\HH^4(X,\K^{\M}_4)=\CH^4(X)=\HH^4(X(\mbr);\mbz/2)$ and $\HH^4(X,\I^4)=\HH^4(X(\mbr);\mbz)$.
	By Theorem\autoref{cohdim}, the functor $\HH^4(X,-)$ is right exact, so it suffices to show that $f$ induces zero on the fourth cohomology group.

	On the summand $\K^{\M}_4$, $f$ induces the Suslin-Hurewicz map, so the induced map $\HH^4(X,\K^{\Q}_4)=\CH^4(X)\to\HH^4(X,\K^{\M}_4)=\CH^4(X)$ is the multiplication by 6. Since $\CH^4(X)$ is 2-torsion, this map is zero.

	Next we deal with the summand $\I^4$. Consider the map $g$ given by the composition
	$$\K^{\M}_4\to\K^{\Q}_4\to \K^{\M}_4\oplus\I^4\to\I^4.$$
	We claim that $g$ is 0. Indeed, since $\K^{\M}_4=\K^{\MW}_4/\eta$ and $\K^{\MW}_4$ is the free strongly $\mba^1$-invariant abelian sheaf generated by $(\Gm)^{\wedge 4}$ (\cite{Mor12A1AT}, Theorem 3.37), we get
	\begin{align*}
		\HHom_{\Ab(\mbr)}(\K^{\M}_4,\I^4) &=\ker(\HHom_{\Ab(\mbr)}(\K^{\MW}_4,\I^4)\xrightarrow{\eta}\HHom_{\Ab(\mbr)}(\K^{\MW}_5,\I^4))\\
		&=\ker(\W(\mbr)\xrightarrow{\eta}\W(\mbr))\\
		&=0.
	\end{align*}
	Now the natural map $\K_4^{\M}\to\K^{\Q}_4$ induces an isomorphism after taking $\HH^4(X,-)$, so the map $f$ is zero on the fourth cohomology group.
\end{proof}

This proposition gives an exact sequence
\[\begin{tikzcd}[ampersand replacement=\&, column sep=small]
	{\HH^3(X,\K^{\Q}_3)} \& {\HH^4(X,A)} \& {\HH^4(X,\pi_3^{\mba^1}(F))} \& 0 \\
	{\CH^3(X)} \& {\HH^4(X(\mbr);\mbz/2)\oplus\HH^4(X(\mbr);\mbz)}
	\arrow["\partial",from=1-1, to=1-2]
	\arrow[equals, from=1-1, to=2-1]
	\arrow[from=1-2, to=1-3]
	\arrow[equals, from=1-2, to=2-2]
	\arrow[from=1-3, to=1-4]
\end{tikzcd}\]
where $\HH^4(X,\K^{\Q}_3)=0$ is because $(\K^{\Q}_3)_{-4}=0$ and the Rost-Schmid complex.

\begin{prop}\label{prop:4.1.4}
	There exists a map $k:\CH^2(X)\to\HH^4(X,\pi_3^{\mba^1}(F))$ such that $V$ is algebraizable if and only if there exists $c_2\in\CH^2(X)$ such that $\bar{\gamma}(c_2)=w_2(V)$ and $k(c_2)$ is the image of $p_1(V)$ and $w_4(V)$ under the map $\HH^4(X,A)\to\HH^4(X,\pi_3^{\mba^1}(F))$.
\end{prop}

\begin{proof}
	By Theorem\autoref{doldwhitney}, $V$ is algebraizable if and only if we can choose $c_2(V)\in\CH^2(X)$ with $\bar{\gamma}(c_2(V))=w_2(V)$ such that the characteristic classes $c_2(V)\in\CH^2(X),w_4(V)\in\CH^4(X),p_1(V)\in\HH^4(X,\I^4)$ can be realized by an algebraic orientable stable vector bundle $X\to\BSL_\infty$.

	Since the cohomological dimension of $X$ is 4 (Theorem\autoref{cohdim}) and $\pi_i^{\mba^1}(F)=0$ for $i=0,1,2$, using the Postnikov tower, we see that the only obstruction to lifting the characteristic classes to a vector bundle lies in the group $\HH^4(X,\pi_3^{\mba^1}(F))$ and is induced by a map
	$$k_3:K(\K^{\M}_2,2)\times K(\K^{\M}_4,4)\times K(\I^{4},4)\to K(\pi_3^{\mba^1}(F),4).$$

	It remains to argue that there is a map $k:\CH^2(X)\to\HH^4(X,\pi_3^{\mba^1}(F))$ such that $k_3$ induces a map of the form
	\begin{align*}
		\CH^2(X)\times\HH^4(X(\mbr);\mbz/2)\times\HH^4(X(\mbr);\mbz)&\to\HH^4(X,\pi_3^{\mba^1}(F))\\
		(c_2,w_4,p_1)&\mapsto (w_4,p_1)-k(c_2)
	\end{align*}
	in the corresponding cohomology groups of $X$. Indeed, since $K(\pi_3^{\mba^1}(F),4)$ is an infinite loop space, we use the equivalence $\Sigma(\mcx\times\mcy)=\Sigma\mcx\vee \Sigma\mcy\vee\Sigma(\mcx\wedge\mcy)$ to get
	\begin{align*}
		&[K(\K^{\M}_2,2)\times K(\K^{\M}_4\oplus\I^4,4),K(\pi_3^{\mba^1}(F),4)]_{\mba^1}\\ =&[K(\K^{\M}_2,2),K(\pi_3^{\mba^1}(F),4)]_{\mba^1}\times[K(\K^{\M}_4\oplus\I^4,4),K(\pi_3^{\mba^1}(F),4)]_{\mba^1}\\ \times&[K(\K^{\M}_2,2)\wedge K(\K^{\M}_4\oplus\I^4,4),K(\pi_3^{\mba^1}(F),4)]_{\mba^1}.
	\end{align*}
	Since $K(\K^{\M}_2,2)\wedge K(\K^{\M}_4\oplus\I^4,4)$ is 5-connected, $$[K(\K^{\M}_2,2)\wedge K(\K^{\M}_4\oplus\I^4,4),K(\pi_3^{\mba^1}(F),4)]_{\mba^1}=0$$
	so there exists a map $k:K(\K^{\M}_2,2)\to K(\pi_3^{\mba^1}(F),4)$ such that $k_3=(\varphi\circ\pr_2)-(k\circ\pr_1)$ where $\varphi:K(\K^{\M}_4\oplus\I^4,4)\to K(\pi_3^{\mba^1}(F),4)$ is the natural map and $\pr_i$ denotes the projections onto the $i$-th component.
\end{proof}

Next, we need to identify the boundary map $\partial$.
Let $\partial_1:\CH^3(X)\to\HH^4(X(\mbr);\mbz/2)$ and $\partial_2:\CH^3(X)\to\HH^4(X(\mbr);\mbz)$ be the composition of $\partial$ and the projections.

\begin{prop}\label{partial2}
	$\partial_2=0$.
\end{prop}

\begin{proof}
	We have $\HH^4(X(\mbr);\mbz)=\mbz^p\oplus(\mbz/2)^q$ (recall that $p$ resp. $q$ is the number of orientable resp. non-orientable connected components of $X(\mbr)$). By Lemma\autoref{lem:3.2.1.2}, $\CH^3(X)$ is 2-torsion, so $\partial_2$ is zero on the summand $\mbz^p$. Next, over non-orientable connected components, we have the relation $p_1=w_2^2$ for any orientable vector bundle. 
	If $\partial_2$ is nonzero, there would be algebraic vector bundles with the same $c_2$ and different $p_1$ over non-orientable components. But this is impossible, since (over non-orientable components) $p_1$ is determined by $c_2$ as $p_1=w_2^2=\bar{\gamma}(c_2)^2$.
\end{proof}

It takes a bit more work to identify $\partial_1$.

\begin{thm}\label{partial1}
	The map $\partial_1:\CH^3(X)\to\HH^4(X(\mbr);\mbz/2)$ coincides with the map $(-)\cup w_1(X)$.
\end{thm}

\begin{proof}
	Let $F_1$ be the fiber of the map
	$$(c_2,c_4):\BSL_\infty\to K(\K^{\M}_2,2)\times K(\K^{\M}_4,4).$$
	Then $\pi_i^{\mba^1}(F_1)=0$ for $i=0,1,2$ and there is an exact sequence
	$$\K^{\Q}_4\xrightarrow{\psi}\K^{\M}_4\to\pi_3^{\mba^1}(F)\to\K^{\Q}_3\to0$$
	where $\psi$ is the Suslin-Hurewicz map. By \cite{röndigs2022endomorphismsprojectiveplaneimage}, the cokernel of $\psi$ is $\K^{\M}_4/6$, so we have a short exact sequence
	$$0\to\K^{\M}_4/6\to \pi_3^{\mba^1}(F)\to\K^{\Q}_3\to0.$$
	Observe that there is a natural map $F\to F_1$ and a commutative diagram
	\[\begin{tikzcd}[ampersand replacement=\&,column sep=small]
		\& {\CH^4(X)\oplus\HH^4(X(\mbr);\mbz)} \&\& \\
		{\CH^3(X)} \& {\HH^4(X,A)} \& {\pi_3^{\mba^1}(F)} \& 0 \\
		{\CH^3(X)} \& {\HH^4(X,\K^{\M}_4/6)} \& {\pi_3^{\mba^1}(F_1)} \& 0 \\
		\& {\CH^4(X)}
		\arrow[equals, from=1-2, to=2-2]
		\arrow["\partial", from=2-1, to=2-2]
		\arrow[equals, from=2-1, to=3-1]
		\arrow[from=2-2, to=2-3]
		\arrow[from=2-2, to=3-2]
		\arrow[from=2-3, to=2-4]
		\arrow[from=2-3, to=3-3]
		\arrow[from=3-1, to=3-2]
		\arrow["{\partial_1}"', from=3-1, to=4-2]
		\arrow[from=3-2, to=3-3]
		\arrow[equals, from=3-2, to=4-2]
		\arrow[from=3-3, to=3-4]
	\end{tikzcd}\]
	where the rows come from long exact sequences in cohomology groups.

	Now the third contraction $\pi_3^{\mba^1}(F_1)_{-3}$ fits into an exact sequence
	$$0\to\Gm/6\to\pi^{\mba^1}_3(F_1)_{-3}\to\mbz\to0.$$
	Since over the site $\Sm_{\mbr}$, taking global sections of abelian sheaves is an exact functor, the sheaf $\mbz$ is a projective object and the above sequence splits.

	Let $C\subseteq X$ be a smooth closed algebraic curve representing a class $[C]\in\CH^3(X)$ and let $i:C\to X$ be the inclusion. Using Rost-Schmid complexes, for any strongly $\mba^1$-invariant sheaf of abelian groups, we can define maps (cf. \cite{Mor12A1AT}, Corollary 5.30)
	$$i_*:\HH^n(C,M_{-3}(N))\to\HH^{n+3}(X,M)$$
	where $N=\omega_{C/\mbr}\otimes i^*\omega_{X/\mbr}^\vee$ is a line bundle.
	Since $C$ is a curve, $\omega_{C/\mbr}$ is topologically trivial, we deduce that $\omega_{C/\mbr}$ is trivial by Theorem\autoref{thm:2.3.0.2}, so $N=i^*\omega_{X/\mbr}^\vee$.
	So we have a commutative diagram with exact rows
	\[\begin{tikzcd}[ampersand replacement=\&]
		{\HH^0(C,\mbz)} \& {\HH^1(C,\Gm/6)} \& {\HH^1(C,\pi_3^{\mba^1}(F_1)_{-3}(N))}\\
		{\HH^{3}(X,\K^{\Q}_3)} \& {\HH^4(X,\K_4^{\M}/6)} \& {\HH^4(X,\pi_3^{\mba^1}(F_1))} \\
		{\CH^3(X)} \& {\CH^4(X)}
		\arrow["{\partial'(N)}", from=1-1, to=1-2]
		\arrow[from=1-1, to=2-1]
		\arrow[from=1-2, to=2-2]
		\arrow[from=2-1, to=2-2]
		\arrow[equals, from=2-1, to=3-1]
		\arrow[equals, from=2-2, to=3-2]
		\arrow["\partial_1", from=3-1, to=3-2]
		\arrow[from=1-2, to=1-3]
		\arrow[from=2-2, to=2-3]
		\arrow[from=1-3, to=2-3]
	\end{tikzcd}\]
	where $\partial'(N)$ is the boundary map coming from the short exact sequence
	$$0\to\Gm/6\to\pi_3^{\mba^1}(F_1)_{-3}(N)\to\mbz\to0.$$

	To describe $\partial'(N)$, we need to understand the $\Gm$-action on $\pi_3^{\mba^1}(F_1)_{-3}$. We fix a non-canonical isomorphism $\pi_3^{\mba^1}(F_1)_{-3}\cong\Gm/6\oplus\mbz$. Using the fact that $\K^{\MW}_1$ is the free strongly $\mba^1$-invariant abelian sheaf generated by the pointed set $\Gm$ (\cite{Mor12A1AT}, Theorem 3.37), we have 
	\begin{align*}
		\intHom(\pi_3^{\mba^1}(F_1)_{-3},\pi_3^{\mba^1}(F_1)_{-3})&=\intHom(\mbz,\mbz)\oplus\intHom(\mbz,\Gm/6)\\
		&\oplus\intHom(\Gm/6,\mbz)\oplus\intHom(\Gm/6,\Gm/6)\\
		&=\mbz\oplus\Gm/6\oplus0\oplus\mbz/6
	\end{align*}
	so any $\Gm$-action on $\pi_3^{\mba^1}(F_1)_{-3}$ factors through a map $\Gm\to\Gm/6$ and has the form
	\begin{align*}
		\Gm\times(\Gm/6\oplus\mbz)&\to \Gm/6\oplus\mbz\\
		(g,(g',n))&\mapsto (g'+an\cdot g,n)
	\end{align*}
	where $a\in\mbz/6$.

	We claim that $\partial'(N)(1)=c_1(N)$. If this is true, we may conclude the proof by using the projection formula. Using the exactness of the top row of the above diagram, it suffices to show that $c_1(N)$ is mapped to 0 in $\HH^1(C,\pi_3^{\mba^1}(F_1)_{-3}(N))$.
	Now we do some explicit computations using the Rost-Schmid complex of the curve $C$. More precisely, we have a commutative diagram 
	\[\begin{tikzcd}[ampersand replacement=\&]
		{\mbr(C)^\times/6} \& {\bigoplus_{x\in C^{(1)}}\mbz/6} \& {\Pic(C)/6} \& 0 \\
		{\mbr(C)^\times/6\oplus\mbz} \& {\bigoplus_{x\in C^{(1)}}\mbz/6} \& {\HH^1(C,\pi_3^{\mba^1}(F_1)_{-3})} \& 0
		\arrow["v", from=1-1, to=1-2]
		\arrow[from=1-1, to=2-1]
		\arrow[from=1-2, to=1-3]
		\arrow[from=1-2, to=2-2]
		\arrow[from=1-3, to=1-4]
		\arrow[from=1-3, to=2-3]
		\arrow["{v(N)}", from=2-1, to=2-2]
		\arrow[from=2-2, to=2-3]
		\arrow[from=2-3, to=2-4]
	\end{tikzcd}\]
	with exact rows given by the Rost-Schmid complexes. Here $\mbr(C)$ is the field of rational functions of $C$. The map $v$ is the map taking valuations. The map $v(N)$ can be described explicitly as follows. Let $s$ be a global section of $N$ with simple zeros $x_1,\cdots,x_r$. Choose uniformizers $t_i\in\mbr(C)^\times$ of the discrete valuation rings $\mco_{C,x_i}$. Then for $x\in C\setminus\{x_1,\cdots,x_i\}$, we have
	$v(N)_{x}(g,n)=v_x(g)$, and if we let $\eta$ be the generic point of $C$, we have
	\begin{align*}
		v(N)_{x_i}(g,n)&=v(N)_{x_i}((g,n)\otimes s|_\eta)\\
		&=v(N)_{x_i}(t_i(g,n)\otimes (s/t_i)|_\eta)\\
		&=v_{x_i}(g+an\cdot t_i)\otimes(s/t_i)|_{x_i}\\
		&= v_{x_i}(g)+an.
	\end{align*}

	$c_1(N)$ is represented by the element $x_1+\cdots+x_r\in\bigoplus_{x\in C^{(1)}}\mbz/6$. If $a$ is odd, then $a-1$ is even and $v(N)(1,1)-(x_1+\cdots+x_n)=(a-1)(x_1+\cdots+x_n)$ lies in the image of $v(N)|_{\mbr(C)^\times/6}$. So $c_1(N)$ is mapped to 0 in $\HH^1(C,\pi_3^{\mba^1}(F_1)_{-3}(N))$. On the other hand, if $a$ is even, then the image of $v(N)$ is the same as the image of $v(N)|_{\mbr(C)^\times/6}$ and in this case $\partial'(N)(1)$ is zero. So it suffices to find an instance where $\partial'(N)(1)$ is nonzero in order to show that $a$ is odd.

	Let $X=\mbr\mbp^4$. It is an affine real algebraic variety by \cite{BCR} Theorem 3.4.4. In this specific case, every topological real vector bundle over $X$ is algebraizable by \cite{algbundonprojspace}.
	By a standard Postnikov tower argument, one can show that the pair $(w_2(V),w_4(V))$ can be arbitrary when $V$ ranges through orientable topological real vector bundles over $X(\mbr)$.
	Since the map $\CH^4(X)\to\HH^4(X,\pi_3^{\mba^1}(F_1))$ is surjective, the obstruction group $\HH^4(X,\pi_3^{\mba^1}(F_1))$ must vanish. Hence $\partial$ is surjective. This shows that $a$ is odd, and the proof is complete.
\end{proof}

\begin{remark}
	The above argument mimics the proof of \cite{AsokFasel2014threefold} Theorem 4.17.
\end{remark}

\begin{remark}
	Since $X(\mbr)$ is a smooth fourfold, one can show that the operation $\alpha\mapsto \alpha\cup w_1(X)$ is the same as the Steenrod square or the Bockstein $\Sq^1=\beta: \HH^3(X(\mbr);\mbz/2)\to\HH^4(X(\mbr);\mbz/2)$.
\end{remark}

\begin{remark}
	It is very likely that $\pi_3^{\mba^1}(F_1)=\pi_3^{\mba^1}(\BGL_3)$. If this is true, then one can describe the $\Gm$-action on $\pi_3^{\mba^1}(F_1)_{-3}$ by looking at vector bundles over a variety homotopy equivalent to $\mbs^3\wedge\Gm^{\wedge 3}$. However, this seems more difficult and the precise action of $\Gm$ is not needed in our argument.
\end{remark}

\begin{coro}
	We have $$\HH^4(X,\pi_3^{\mba^1}(F))=\HH^4(X(\mbr);\mbz)\oplus\coker(\CH^3(X)\xrightarrow{(-)\cup w_1(X)}\CH^4(X)).$$
\end{coro}

Now we can prove one of our main theorems.

\begin{thm}\label{dim4}
	Let $X$ be an affine smooth real algebraic variety of dimension 4 and assume that $X(\mbr)$ is compact. Let $M_1,\cdots,M_p$ be the orientable connected components of $X(\mbr)$, and let $N_1,\cdots,N_q$ be the non-orientable connected components. Then there exist group homomorphisms $k_i:\CH^2(X)\to\HH^4(M_i;\mbz)$ for $1\leq i\leq p$ and maps $k_j':\CH^2(X)\to\HH^4(N_j;\mbz/2)$ for $1\leq j\leq q$ such that the following holds:

	A topological real vector bundle $V$ over $X(\mbr)$ is algebraizable if and only if there exists $c_1\in\CH^1(X),c_2\in\CH^2(X)$ and a class $\alpha\in\HH^3_{\alg}(X;\mbz/2)$ such that
	\begin{enumerate}
		\item $\bar{\gamma}(c_1)=w_1(V),\bar{\gamma}(c_2)=w_2(V)$,
		\item for each $1\leq i\leq p$, $$p_1(V|_{M_i})=k_i(c_2+c_1^2),$$
		\item for each $1\leq j\leq q$, $$(w_4(V)+w_1(V)w_3(V))|_{N_j}=\alpha|_{N_j}\cup w_1(N_j)+k'_j(c_2+c_1^2).$$
	\end{enumerate}
\end{thm}

\begin{proof}
	Let $V'=V\oplus\det V$. Then $V$ is algebraizable if and only if $w_1(V)$ is algebraic and $V'$ is algebraizable by Theorem\autoref{thm:2.3.0.2}. We have $w_2(V')=w_2(V)+w_1(V)^2$, $w_4(V')=w_4(V)+w_1(V)w_3(V)$, and $p_1(V')=p_1(V)$. By replacing $V$ by $V'$ and replacing the corresponding characteristic classes, we may assume that $V$ is orientable.
	
	Let $k:\CH^2(X)\to\HH^4(X,\pi_3^{\mba^1}(F))$ be the map given in Proposition\autoref{prop:4.1.4}.
	For any connected component $Y\subseteq X(\mbr)$ of $X(\mbr)$, let $k_Y$ be the composition
	$$\CH^2(X)\xrightarrow{k}\HH^4(X,\pi_3^{\mba^1}(F))\to\HH^4(Y;\mbz)$$
	where the second map is the projection. Similarly, we fix a section
	$$i:\coker(\CH^3(X)\xrightarrow{(-)\cup w_1(X)}\CH^4(X))\to\CH^4(X)$$
	and let $k'_Y$ be the composition
	$$\CH^2(X)\xrightarrow{k}\HH^4(X,\pi_3^{\mba^1}(F))\to\coker(\CH^3(X)\to\CH^4(X))\xrightarrow{i}\CH^4(X)\to\HH^4(Y;\mbz/2)$$
	where the second and the fourth maps are the projections. For $1\leq i\leq p$ (resp. $1\leq j\leq q$), set $k_i=k_{M_i}$ (resp. $k'_j=k'_{N_j}$).
	The maps $k_i$ are group homomorphisms since $c_2$ and $p_1$ of orientable vector bundles over an orientable connected component are additive.

	Using Proposition\autoref{prop:4.1.4} while keeping in mind that we have identified $\partial$ (Proposition\autoref{partial2} and Theorem\autoref{partial1}), we see that it suffices to prove that the obstructions arising from $k_{N_j}$ and $k'_{M_i}$ vanish automatically. Indeed, as in the proof of Proposition\autoref{partial2}, the obstruction given by $k_{N_j}$ is the relation $p_1=w_2^2$ which is always satisfied since the pair $(w_2,p_1)$ comes from a vector bundle. Similarly, over $M_i$, $c_4=w_4$ is determined by $c_2$ and $p_1$ by the relation $p_1\equiv\mathfrak{P}(w_2)+j(w_4)\mod 4$ where $\mathfrak{P}$ is the Pontryagin square and $j:\HH^4(S;\mbz/2)\to\HH^4(S;\mbz/4)$ is induced by the inclusion $\mbz/2\to\mbz/4$. This relation is again automatic since the triple $(w_2,w_4,p_1)$ comes from a vector bundle.
\end{proof}

\begin{remark}
	The maps $k_i$ can be described as follows. Consider the scheme $X_{\mbc}=X\otimes_{\mbr}\mbc$. It is essentially smooth over $\mbc$ and $X_{\mbc}(\mbc)$ is homeomorphic to $X(\mbr)$. Let $p:X_{\mbc}\to X$ denote the natural map. Then $k_i$ is given by the composition
	$$\CH^2(X)\xrightarrow{p^*}\CH^2(X_{\mbc})\xrightarrow{\cl}\HH^4(X_{\mbc}(\mbc);\mbz)=\HH^4(X(\mbr);\mbz)\to\HH^4(M_i;\mbz).$$
\end{remark}

\begin{remark}
	Note that the maps $k'_j$ are not group homomorphisms. In fact, if $V$ and $V'$ are two orientable vector bundles, then $w_4(V\oplus V')=w_4(V)+w_4(V')+w_2(V)w_2(V')$, so the induced maps
	$$k'_j:\CH^2(X)\to\HH^4(N_j;\mbz/2)/\beta\HH^3_{\alg}(N_j;\mbz/2)$$
	are quadratic in some sense.
	Unfortunately, the author is unable to identify the maps $k'_j$ explicitly.
\end{remark}

\begin{remark}
	In fact, for a ``sufficiently generic'' real algebraic variety $X$ homeomorphic to $\mbs^4$, algebraic vector bundles over $X$ are stably trivial (\cite{bochnak1989vector}, Theorem 1.1).
\end{remark}

\subsection{Some Applications}

\subsubsection{Orientable Compact Fourfolds of the Form $M\times\mbs^1$}

Here we study a family of orientable examples. We start with the following lemma concerning Chow groups.

Let $\mbs^1$ denote the real algebraic variety defined by $\Spec\mbr[x,y]/(x^2+y^2-1)$.

\begin{lemma}\label{lem:4.2.1}
	Let $Y$ be any affine real algebraic variety.
	Let $p:Y_{\mbc}\to Y$ be the complexification map.
	Then
	$$\CH^i(Y\times\mbs^1)=\CH^{i-1}(Y)/p_*(\CH^{i-1}(Y_{\mbc}))\oplus\CH^i(Y).$$
\end{lemma}

\begin{proof}
	Let $\infty=\{x^2+y^2=0\}=\Spec\mbc\in\mbp_{\mbr}^1$ be the point at infinity and let $S=\mbp^1_{\mbr}\setminus \{\infty\}$. Then $S=\mbr[x,y]/(x^2+y^2-1)$ and $\mbs^1$ is the real algebraic variety defined by $S$.

	We first compute the Chow groups of $Y\times S$. The projective bundle theorem gives
	$$\CH^i(Y\times \mbp^1_{\mbr})=\CH^{i-1}(Y)\oplus\CH^i(Y)$$
	and we have an exact sequence
	$$\CH^{i-1}(Y\times\{\infty\})\to\CH^{i-1}(Y)\oplus\CH^i(Y)\to\CH^i(Y\times S)\to 0.$$
	One easily sees that the first map coincides with $(p_*,0)$ so $$\CH^i(Y\times S)=\CH^{i-1}(Y)/p_*(\CH^{i-1}(Y_{\mbc}))\oplus\CH^i(Y).$$

	Now, since all points in $\mbp_{\mbr}^1\setminus\mbs^1$ are rationally equivalent, the quotient $\CH^i(Y\times S)\to\CH^i(Y\times\mbs^1)$, obtained from removing closed subsets of the form $X\times\{x\}$ with $x\in \mbp_{\mbr}^1\setminus\mbs^1$, is an isomorphism. We conclude by taking the colimit.
\end{proof}

Now we can prove one of the main theorems.

\begin{thm}\label{mainexam}
	Let $Y$ be a smooth affine real algebraic variety of dimension 3 such that $Y(\mbr)$ is compact and orientable. Let $X=Y\times\mbs^1$, where $\mbs^1$ is the standard circle. Then a topological real vector bundle $V$ over $X$ is algebraizable if and only if $w_*(V)\in\HH^*_{\alg}(X;\mbz/2)$ and $p_1(V)=0$.
\end{thm}

\begin{proof}
	Recall that $\CH^3(Y)=\mbz/2$ by \cite{CTS96}, $\CH^2(Y)$ is 2-torsion by Lemma\autoref{lem:3.2.1.2}, $\CH^1(Y)$ is 2-torsion by Theorem\autoref{thm:2.3.0.2}, and $\CH^0(Y)=\mbz$ obviously. Also recall that if we let $p:Y_{\mbc}\to Y$ be the complexification map, then the map $p_*p^*$ is multiplication by 2. By Lemma\autoref{lem:4.2.1}, we see that the Chow groups of $X=Y\times\mbs^1$ are 2-torsion in codimension $\geq 1$.

	Let $k:\CH^2(X)\to\HH^4(X(\mbr);\mbz)$ be the map given in Theorem\autoref{dim4}. Since $\CH^2(X)$ is 2-torsion and $\HH^4(X(\mbr);\mbz)=\mbz$, the map $k$ is zero. We conclude by applying Theorem\autoref{dim4}.
\end{proof}

\subsubsection{$\tildeK_0$ of Compact Fourfolds}

Again, let $X$ be a smooth affine real algebraic variety of dimension 4 such that $X(\mbr)$ is compact.
Here we give a description of $\tildeK_0(X)$.

Consider the abelian group $G=1\oplus\HH^1_{\alg}(X;\mbz/2)\oplus\CH^2(X)$ (recall that $\CH^1(X)=\HH^1_{\alg}(X;\mbz/2)$) where the group operation is given by multiplication. Then there is a group homomorphism
$$(c_1,c_2):\tildeK_0(X)\to G.$$
Also set
$$d=\dim_{\mbz/2} w_1(X)\cup\HH^3_{\alg}(X;\mbz/2).$$

\begin{prop}\label{prop:4.2.3}
	There is an exact sequence
	$$0\to(\mbz/2)^d\to\tildeK_0(X)\to G\to 0.$$
\end{prop}

\begin{proof}
	We first argue that the map $\tildeK_0(X)\to G$ is surjective. Indeed, the morphism
	$$(c_1,c_2):\BGL_\infty\to K(\K^{\M}_1,1)\times K(\K^{\M}_2,2)$$
	induces isomorphisms on $\pi_1^{\mba^1}$ and $\pi_2^{\mba^1}$, so the primary obstruction (which is also the only obstruction) of lifting a pair of classes $(c_1,c_2)$ to a vector bundle lies in $\HH^4(X,\K^{\Q}_3)=0$. So $\tildeK_0(X)\to G$ is surjective.

	Next, we look at its kernel. By definition, its kernel contains algebraic vector bundles $V$ with $c_1(V)=0$ and $c_2(V)=0$. By Theorem\autoref{dim4}, $p_1(V)=0$ since $c_2(V)=0$, and $w_4(V)$ can take any value in $w_1(X)\cup\HH^3_{\alg}(X;\mbz/2)\subseteq \CH^4(X)=\HH^4(X(\mbr);\mbz/2)$. So the kernel is isomorphic to $(\mbz/2)^d$.
\end{proof}

If $X(\mbr)$ is orientable, then we can say more about the groups $\CH^2(X)$ and $\tildeK_0(X)$.

\begin{prop}
	Assume that $X(\mbr)$ is orientable. Then the group homomorphism
	$$(\bar{\gamma},k):\CH^2(X)\to\HH^2(X(\mbr);\mbz/2)\times\HH^4(X(\mbr);\mbz)$$
	is injective. Here $k$ is the map given in Theorem\autoref{dim4}.
\end{prop}

\begin{proof}
	Take $c\in\CH^2(X)$ such that $\bar{\gamma}(c)=0$ and $k(c)=0$. By Theorem\autoref{dim4}, there exists an orientable algebraic vector bundle $V$ with $c_2(V)=c$ and $p_1(V)=0$. By Theorem\autoref{doldwhitney}, $V$ is topologically trivial, and by Theorem\autoref{thm:2.3.0.2}, $V$ is algebraically trivial. So $c=c_2(V)=0$.
\end{proof}

Using this proposition, we may write $\CH^2(X)=(\mbz/2)^s\oplus\mbz^t$ for some integers $s,t$.
Let $$b_1=\dim_{\mbz/2}\HH^1_{\alg}(X;\mbz/2)$$ and 
$$\delta=\dim_{\mbz/2}\{v\in\HH^1_{\alg}(X;\mbz/2)| v^2=0\}.$$

\begin{thm}\label{thm:4.2.5}
	Assume that $X(\mbr)$ is orientable. Then we have an isomorphism
	$$\tildeK_0(X)\cong \mbz^t\oplus(\mbz/2)^{s-b_1+2\delta}\oplus(\mbz/4)^{b_1-\delta}.$$
\end{thm}

\begin{proof}
	By Proposition\autoref{prop:4.2.3}, we have $\tildeK_0(X)\cong G$. Since $\HH^1_{\alg}(X;\mbz/2)$ is 2-torsion, cup products of elements in $\HH^1_{\alg}(X;\mbz/2)$ lie in the $(\mbz/2)^s$ summand of $\CH^2(X)$. So $\mbz^t$ is a direct summand of $G$. Its complement is the group $1\oplus\HH^1_{\alg}(X;\mbz/2)\oplus(\mbz/2)^s$ with group operation given by multiplication. The rest follows easily by counting.
\end{proof}

\printbibliography[title={References}]

\end{document}